\documentclass[a4paper,10pt, reqno]{amsart}

\usepackage{stmaryrd} 
\usepackage[numbers]{natbib}
\usepackage{float}
\usepackage{graphicx}
\usepackage{url}
\usepackage{amssymb,amsmath,amsthm,verbatim,longtable,latexsym}
\usepackage{mathabx}
\usepackage[all,ps]{xy}
\usepackage{color}
\usepackage{hyperref}
\usepackage{times}

\newcommand{\cA}{\mathcal{A}}
\newcommand{\cB}{\mathcal{B}}
\newcommand{\cC}{\mathcal{C}}
\newcommand{\cD}{\mathcal{D}}
\newcommand{\cE}{\mathcal{E}}
\newcommand{\cH}{\mathcal{H}}
\newcommand{\cX}{\mathcal{Y}}
\newcommand{\cY}{\mathcal{Z}}

\newcommand{\rA}{\mathrm A}
\newcommand{\rB}{\mathrm B}
\newcommand{\rBB}{\mathrm B} 
\newcommand{\rC}{\mathrm C}
\newcommand{\rCC}{\mathrm C} 
\newcommand{\rD}{\mathrm D}
\newcommand{\rE}{\mathrm E}
\newcommand{\rF}{\mathrm F}
\newcommand{\rG}{\mathrm G}
\newcommand{\rH}{\mathrm H}
\newcommand{\rL}{\mathrm L} 
\newcommand{\rM}{\mathrm M}
\newcommand{\rN}{\mathrm N}
\newcommand{\rS}{\mathrm S}
\newcommand{\rT}{\mathrm T}
\newcommand{\rU}{\mathrm U}
\newcommand{\rV}{\mathrm V}
\newcommand{\rW}{\mathrm W}

\newcommand{\bA}{\mathbb A}
\newcommand{\bB}{\mathbb B}
\newcommand{\bC}{\mathbb C}
\newcommand{\bD}{\mathbb D}
\newcommand{\bG}{\mathbb G}
\newcommand{\bL}{\mathbb L}
\newcommand{\bS}{\mathbb S}
\newcommand{\bT}{\mathbb T}
\newcommand{\bV}{\mathbb V}

\newcommand{\cMix}{\mathsf{Mix}} 
\newcommand{\cDist}{\mathsf{Dist}} 
\newcommand{\cSet}{\mathsf{Set}} 
\newcommand{\cMod}{\mbox{-}\mathsf{Mod}}
\newcommand{\cComod}{\mbox{-}\mathsf{Comod}}

\newcommand{\Aop}{{A^\mathrm{op}}}
\newcommand{\Ae}{{A^\mathrm{e}}}

\renewcommand\epsilon{\varepsilon}
\renewcommand\phi{\varphi} \newcommand{\adj}{\Lambda}
\newcommand{\lmapsto}{\longmapsto}

\newcommand{\op}{\mathrm{op}}

\newcommand{\idty}{\mathsf{id}}

\newcommand{\lact}{\smalltriangleright}
\newcommand{\ract}{\smalltriangleleft}
\newcommand{\blact}{\blacktriangleright}
\newcommand{\bract}{\blacktriangleleft}

\newcommand{\rmref}[1]{{\rm (}\ref{#1}{\rm )}}

\numberwithin{equation}{section}

\newcommand{\STRINGDIAGRAM}{\xymatrix@R=10pt@C=10pt@H=0pt@W=0pt@M=0pt}

\newcommand{\bprod}{\ar@{-}[dd] &
\ar@{-}@(d,r)@/^3mm/[dl]}
\newcommand{\bkopr}{\ar@{-}[uu]
&\ar@{-}@(u,l)@/_3mm/[ul]}

\newcommand{\bmult}{\ar@{-}@(d,r)@/^2mm/}
\newcommand{\bkomu}{\ar@{-}@(u,r)@/_1.7mm/}
\newcommand{\bbigkomu}{\ar@{-}@(u,r)@/_4mm/}

\newcommand{\FERMION}{\ar@{-}@(d,u)}
\newcommand{\FERMIONddl}{\ar@{-}@(d,u)[ddl]}
\newcommand{\FERMIONddddr}{\ar@{-}@(d,u)[ddddr]}
\newcommand{\FERMIONddddrr}{\ar@{-}@(d,u)[ddddrr]}
\newcommand{\FERMIONddddl}{\ar@{-}@(d,u)[ddddl]}
\newcommand{\FERMIONddddll}{\ar@{-}@(d,u)[ddddll]}
\newcommand{\FERMIONddddlll}{\ar@{-}@(d,u)[ddddlll]}
\newcommand{\BOSON}{\ar@{~}}

\newtheorem{thm}{Theorem}[section]
\newtheorem{prop}[thm]{Proposition}
\newtheorem{lem}[thm]{Lemma}

\theoremstyle{definition}
\newtheorem{exa}[thm]{Example}
\newtheorem{defn}[thm]{Definition}
\newtheorem*{thmm}{Theorem}

\theoremstyle{remark}

\title{Cyclic homology arising from adjunctions}

\author{Niels Kowalzig} \address{Universit\`a di Napoli
Federico II, Dipartimento di Matematica e Applicazioni,
P.le Tecchio 80, 80125 Napoli, Italia}
\email{niels.kowalzig@unina.it}

\author{Ulrich Kr\"ahmer} \address{University of
Glasgow, School of Mathematics and Statistics, 15
University Gardens, G12 8QW Glasgow, UK}
\email{ulrich.kraehmer@glasgow.ac.uk}

\author{Paul Slevin} \address{University of
Glasgow, School of Mathematics and Statistics, 15
University Gardens, G12 8QW Glasgow, UK}
\email{p.slevin.1@research.gla.ac.uk}

\begin{document}
\begin{abstract}
Given a monad and a
comonad, one obtains a distributive law between them
from lifts of one through an adjunction for the other.
In particular, this yields for any bialgebroid the
Yetter-Drinfel'd distributive law between the comonad
given by a module coalgebra and the monad given by a
comodule algebra. It is this self-dual setting that
reproduces the cyclic homology of associative and of
Hopf algebras in the monadic framework of B\"ohm and \c
Stefan.  In fact, their approach generates two
duplicial objects and morphisms between them which are
mutual inverses if and only if the duplicial objects
are cyclic. A $2$-categorical perspective on the process
of twisting coefficients is provided and the r\^ ole of
the two notions of bimonad studied in the literature is
clarified.
\end{abstract}

\maketitle

\tableofcontents

\section{Introduction}
\subsection{Background and aim}
The Dold-Kan correspondence generalises chain complexes
in abelian categories to general simplicial objects,
and thus homological algebra to homotopical algebra.
The classical homology theories
defined by an augmented algebra (such as group, Lie
algebra, Hochschild, de Rham and Poisson homology)
become expressed as the homology of suitable comonads
$\bT$, defined via simplicial objects
$\rCC_\bT(\rN,\rM)$ obtained from the bar
construction (see, {\em e.g.},~\cite{MR1269324}).

Connes' cyclic homology created a new paradigm
of homology theories defined in terms of mixed
complexes \cite{MR883882,MR826872}. The homotopical
counterparts are cyclic
\cite{Con:CCEFE} or more generally duplicial
objects \cite{MR826872,MR885102}, and
B\"ohm and \c Stefan \cite{MR2415479} showed how
$\rCC_\bT(\rN,\rM)$ becomes duplicial in the
presence of a second comonad $\bS$ compatible in
a suitable sense with $\rN,\rM$ and $\bT$.

The aim of the present article is to study
how the cyclic homology of associative algebras and
of Hopf algebras in the original sense of Connes and
Moscovici \cite{MR1657389}
fits into this monadic formalism, extending
the construction from \cite{MR2803876}, and to clarify
the r\^ole of different notions of bimonad in this
generalisation.

\subsection{Distributive laws arising from adjunctions}
Inspired by \cite{MR3175323,MR3020336} we begin by
describing the relation of
distributive laws between (co)\-mo\-nads and of lifts of
one of them through an adjunction for the other.  In
particular, we have:

\begin{thmm}
Let $\rF \dashv \rU$ be an adjunction,
$\bB:=(\rB,\mu,\eta)$,
$\rB=\rU \rF$, and
$\bT=(\rT,\Delta,\varepsilon)$,
$\rT= \rF\rU$, be the associated (co)monads,
and  $\bS=(\rS,\Delta^\rS,\varepsilon^\rS)$ and
$\bC=(\rC,\Delta^\rC,\varepsilon^\rC)$ be comonads
with a lax isomorphism
$\Omega \colon \rC\rU \rightarrow \rU\rS$,
$$
	\xymatrix@C=5mm@R=4mm{
	\cB \ar@/^2mm/[dd]^\rU \ar@{}[dd]|{\dashv}
	\ar[rr]^{\rS} & &
	\cB \ar@/^2mm/[dd]^\rU \ar@{}[dd]|{\dashv}\\
	\\ \cA
	\ar@/^2mm/[uu]^ \rF \ar[rr]_{\rC} & & \cA
	\ar@/^2mm/[uu]^
	\rF}
$$
If $ \Lambda \colon \rF\rC
\rightarrow \rS\rF$ corresponds under the adjunction to
$\Omega\rF \circ \rC \eta \colon \rC \rightarrow
\rU\rS\rF$,
where $ \eta $ is the unit of $\rB$, then
the following are (mixed) distributive laws:
\begin{gather*}
	\xymatrix{\theta \colon \rB\rC=\rU\rF\rC
	\ar[r]^-{\rU \adj} & \rU\rS\rF \ar[r]^-{\Omega^{-1}\rF}
	& \rC\rU\rF=\rC\rB,} \\
	\xymatrix{\chi \colon
	\rT\rS=\rF\rU\rS \ar[r]^-{\rF \Omega^{-1}} & \rF\rC\rU
	\ar[r]^-{\adj \rU} & \rS\rF\rU=\rS\rT.}
\end{gather*}
\end{thmm}

See Theorem~\ref{arise} on p.~\pageref{arise} for a more
detailed statement. For Eilenberg-Moore adjunctions
($\cB=\cA^\bB$),
such lifts $\bS$ of a given comonad $\bC$ correspond bijectively
to mixed distributive laws between $\bB$ and $\bC$
(a dual statement holds for coKleisli adjunctions
$\cA=\cB_\bT$),
{\em cf.}~Section~\ref{exceptional}.

Sections~\ref{distlawssec}--\ref{duplobjsec}
contain various technical results that we
would like to add to the theory developed in
\cite{MR2415479}, while the final two
Sections~\ref{brugsec} and~\ref{wisbimonad} discuss
examples.

First, we further develop the $2$-categorical viewpoint
of \cite{MR2956318}, interpreting
the comparison functor from $\cB$ to the
Eilenberg-Moore category $\cA^\bB$ of $\bB$ as a $1$-cell in the
$2$-category of mixed distributive laws, and
the passage from mixed distributive laws between
$\bB,\bC$ to distributive laws between $\bT,\bS$ in the
case of an Eilenberg-Moore adjunction as the
application of a $2$-functor (Sections~\ref{twocatone}
and \ref{twocattwo}).

Secondly, Section~\ref{galoismapsct} describes how different
lifts $\rS,\rV$ of a given functor $\rC$ are related by
a generalised Galois map $\Gamma^{\rS,\rV}$ that will be used
in subsequent sections.

\subsection{Coefficients}
In Section~\ref{coeffsec}, we
discuss left and right $ \chi $-coalgebras $\rN$
respectively $\rM$ that serve as coefficients of cyclic
homology.

The structure of right $ \chi $-coalgebras
is easily described in terms of
$\bC$-coalgebra structures on $\rU\rM$
(Proposition~\ref{chicoalgprop}).
In the example from \cite{MR2803876} associated to a
Hopf algebroid $H$, these are simply
right $H$-modules and left $H$-comodules,
see Section~\ref{coeffsforhopf} below.

The structure of left $ \chi $-coalgebras is more
intricate. In the Hopf algebroid example, we present
a construction from Yetter-Drinfel'd modules, but we do
not have an analogue of Proposition~\ref{chicoalgprop}
which characterises left $ \chi $-coalgebras in
general. The Yetter-Drinfel'd condition is
necessary for the well-definedness of the left $ \chi
$-coalgebra structure, but not for that of the
resulting duplicial object, see again
Section~\ref{coeffsforhopf}.

The remainder of Section~\ref{coeffsec} explains the structure of
entwined $ \chi $-coalgebras, which in the Hopf algebroid case
are given by Hopf modules; these are homologically
trivial (Proposition~\ref{trivcontract}) and can be
also interpreted as $1$-cells to respectively from the
trivial distributive law (Propositions~\ref{triv1cell}
and~\ref{triv}). One reason for discussing them is to
point out that general $ \chi $-coalgebras can not be
reinterpreted as $1$-cells.

\subsection{Duplicial objects}
Section~\ref{duplobjsec} recalls the construction of
duplicial objects.  We emphasize the self-duality of the
situation by defining in fact two duplicial objects $
\rCC_\bT(\rN,\rM) $ and $ \rCC_\bS^\op(\rN,\rM) $,
arising from bar resolutions using
$\bT$ respectively $\bS$. There is a canonical pair 
of morphisms of duplicial
objects between these which are mutual inverses if and
only if the two objects are cyclic
(Proposition~\ref{cyc}).

Furthermore, we describe in Section~\ref{twistsec} the
process of twisting a pair of coefficients $\rM,\rN$
by what we called a factorisation in \cite{2}. This is
motivated by the example of the twisted cyclic homology of an
associative algebra \cite{MR1943179} and constitutes
our main application of the $2$-categorical language.

\subsection{Hopf monads}
One of our motivations in this project is to
understand how various notions of bimonads studied
in the literature lead to examples of the above theory
that generalise known ones arising from bialgebras and
bialgebroids.

All give rise to
distributive laws, but it seems to us that opmodule
adjunctions over opmonoidal adjunctions as studied
recently by Aguiar and Chase \cite{MR3020336} are the
underpinning of the cyclic homology theories from noncommutative
geometry: such adjunctions are associated to opmonoidal
adjunctions
$$
	\xymatrix{\cE \ar@/^{0.5pc}/[rr]^-\rH
\ar@{}[rr]|{\perp}&& \ar@/^0.5pc/[ll]^-\rE \cH },
$$
so here $\cH$ and $\cE$ are monoidal
categories,
$\rE$ is a strong monoidal functor and
$\rH$ is an opmonoidal functor, see Section~\ref{opmoduleadj}.
In the key example,
$\cH$ is the category ${H\cMod}$ of modules over a
bialgebroid $H$ and $\cE$ is the category of
bimodules over the base algebra $A$ of $H$. In the
special case of the cyclic homology of an associative
algebra $A$, we have $\cH=\cE$ and
$\rH=\rE=\idty$, so this adjunction is irrelevant.
Now the actual opmodule adjunctions defining cyclic
homology are
formed by an $\cH$-module category $\cB$ and an
$\cE$-module category $\cA$. In the example, one can
pick any $H$-module coalgebra $C$ and any $H$-comodule
algebra $B$, take $\cB$ to be the category ${B\cMod}$ of
$B$-modules, $\cA$ be the category $A\cMod$ of $A$-modules, and
the pair of comonads $\bS,\bC$ is given by $C \otimes_A
-$. To obtain the cyclic homology of an associative
algebra one takes $\cB$ to be the category of
$A$-bimodules (or rather right $\Ae$-modules).
Another very natural
example is given by a quantum homogeneous space
\cite{MR1710737}, where $A=k$ is commutative, $H$ is a
Hopf algebra, $B$ is a left coideal subalgebra
and $C:=A/AB^+$ where $B^+$ is the kernel
of the counit of $H$ restricted to $B$. So here the
distributive law arises from the fact that $B$ admits a
$C$-Galois extension to a Hopf algebra $H$; following,
\emph{e.g.,}~\cite{MR1877862} we call $(B,C)$ a Doi-Koppinen
datum.

Bimonads in the sense of Mesablishvili and
Wisbauer also provide examples of the theory
considered. There is no monoidal structure required on the
categories involved, but instead we have
$\rB=\rC$, see Section~\ref{wisbimonad}.
At the end of the paper we give an example of
such a bimonad which is not related to bialgebroids and
noncommutative geometry, but indicates potential
applications of cyclic homology
in computer science.

\bigskip

\noindent {\bf Acknowledgements.}
N.~K.~acknowledges support by UniNA and Compagnia di San Paolo in the framework
of the program STAR 2013,
U.~K.~by the EPSRC grant EP/J012718/1 and
the Polish Government Grant 2012/06/M/ST1/00169,
and P.~S.~by an EPSRC Doctoral
Training Award. We would like to thank
Gabriella B\"ohm, Steve Lack, Tom Leinster, and Danny
Stevenson for
helpful suggestions and discussions.

\section{Distributive laws}\label{distlawssec}

\subsection{Distributive laws}\label{distributive} We
assume the reader is familiar with
(co)monads and their (co)al\-ge\-bras (see,
\emph{e.g.,}~\cite{MR1712872}),
but we briefly recall the notions of
(co)lax morphisms and distributive laws, see,
\emph{e.g.,}~\cite{MR2094071} for more background.

\begin{defn} Let $\bB = (\rB, \mu^\rB, \eta^\rB)$
and $\bA = (\rA,\mu^\rA,\eta^\rA)$ be monads on
categories $\cC$ respectively $\cD$, and let $\Sigma \colon
\cC \rightarrow \cD$ be a functor. A natural transformation $
\sigma \colon \rA \Sigma \rightarrow \Sigma \rB$ is called a
\emph{lax morphism of monads} if the two diagrams
$$
\xymatrix{ \rA\rA \Sigma  \ar[d]_-{\mu^\rA \Sigma}
\ar[r]^-{\rA \sigma} & \rA \Sigma \rB \ar[r]^-{\sigma
\rB} & \Sigma \rB\rB \ar[d]^-{\Sigma \mu^\rB }  \\ \rA
\Sigma \ar[rr]_-{\sigma } && \Sigma \rB }
\quad\quad\quad \xymatrix{ \Sigma \ar[r]^-{\eta^\rA
\Sigma} \ar[dr]_-{\Sigma \eta^\rB} & \rA \Sigma
\ar[d]^-{\sigma} \\ & \Sigma \rB } $$ commute. We
denote this by $ \sigma \colon \bA \Sigma \rightarrow \Sigma
\bB$.  \end{defn}

Analogously, one defines \emph{colax morphisms} $
\sigma \colon \Sigma \bA \rightarrow \bB \Sigma $,
where $ \Sigma \colon \cD \rightarrow \cC$ and $\bA,\bB$
are as before, and (co)lax morphism of comonads.

\begin{defn} A \emph{distributive law}
$ \chi \colon \bA \bB \rightarrow \bB \bA$ between monads
$\bA,\bB$ is a natural transformation $ \chi
\colon \rA\rB \rightarrow \rB\rA$ which is both a lax and a colax
morphism of monads.
\end{defn}

Analogously, one defines distributive laws between
comonads and \emph{mixed distributive law}
\cite{MR0323864} between
monads and comonads.

\subsection{The $2$-categories $\cDist$ and $\cMix$}
Since this will simplify the presentation of some
results, we turn comonad and mixed distributive laws into
the $0$-cells of $2$-categories
$\cDist$ respectively $\cMix$.  This
closely follows Street \cite{MR0299653}, see also
\cite{2}:

\begin{defn}
We denote by $\cDist$ the $2$-category whose
\begin{enumerate}
\item $0$-cells are quadruples $(\cB, \chi, \bT, \bS)$
where $\chi \colon \bT  \bS\rightarrow \bS\bT$ is a
comonad distributive law on a category $\cB$,
\item $1$-cells $(\cB, \chi, \bT, \bS) \rightarrow (\cD,
\tau, \bG, \bC)$ are triples $(\Sigma,
\sigma, \gamma)$, where $\Sigma \colon \cB \rightarrow \cD$ is
a functor, $\sigma \colon \bG\Sigma \rightarrow \Sigma
\bT$ is a lax morphism of comonads and $\gamma
\colon \Sigma \bS \rightarrow \bC \Sigma$ is a
colax morphism of comonads satisfying the Yang-Baxter
equation, {\em i.e.},
$$
	\xymatrix@R=0.5em{ & \Sigma \rT\rS \ar[r]^-{\Sigma
\chi} & \Sigma \rS\rT \ar[dr]^-{\gamma  \rT} & \\
\rG\Sigma \rS \ar[dr]_-{\rG \gamma} \ar[ur]^-{\sigma
\rS} & & & \rC \Sigma \rT \\ & \rG\rC\Sigma
\ar[r]_-{\tau \Sigma} & \rC\rG \Sigma \ar[ur]_-{\rC
\sigma} & }
$$
commutes, and
\item $2$-cells $(\Sigma,
\sigma, \gamma) \Rightarrow (\Sigma', \sigma', \gamma')$
are natural transformations $\alpha \colon \!\Sigma \rightarrow
\Sigma'$ for which the diagrams
$$
	\xymatrix{
	\rG\Sigma \ar[d]_-\sigma \ar[r]^-{\rG \alpha}
	& \rG \Sigma '
	\ar[d]^-{\sigma '} \\ \Sigma \rT \ar[r]_-{\alpha \rT } &
	\Sigma' \rT }
	\quad \quad \quad
	\xymatrix{ \Sigma \rS
	\ar[r]^-{\alpha \rS} \ar[d]_-\gamma & \Sigma' \rS
	\ar[d]^-{\gamma '} \\ \rC \Sigma \ar[r]_-{\rC \alpha}& \rC
	\Sigma' }
$$
commute.
\end{enumerate}
\end{defn}

In the sequel, we will denote $1$-cells
diagrammatically as:
$$
\xymatrix@C=1.5em{ & \cB
\ar@{}[r]^{\bT}="a" \ar@{}[l]_{\bS}="b"
\ar[d]^-{(\Sigma, \sigma, \gamma)} & \\ & \cD
\ar@{}[r]_{\bG}="c" \ar@{}[l]^{\bC}="d" &
\ar @{.}@/_/ "a";"b" |{\chi} \ar @{.}@/^/ "c";"d"
|{\tau} }
$$

In a similar way, we define the $2$-category
$\cMix$
of mixed distributive laws.

\subsection{Distributive laws arising from
adjunctions}\label{adj} The topic of this paper is
distributive laws that are compatible in a
specific way with an adjunction for one of the involved
comonads: let
$\bB=(\rB,\mu,\eta)$ be a monad on a category $\cA$.  Suppose
$$
\xymatrix{\cA \ar@/^{0.5pc}/[rr]^-\rF
\ar@{}[rr]|{\perp}&& \ar@/^0.5pc/[ll]^-\rU \cB }
$$
is an adjunction for $\bB$, that is, $\rB=\rU\rF$,
and let $\bT := (\rT,\Delta, \varepsilon )$
with $\rT:=\rF\rU$ be the induced comonad on $\cB$.

\begin{defn}\label{lift}
If $\rS \colon \cB \rightarrow \cB$ and
$\rC \colon
\cA \rightarrow \cA$ are endofunctors for which the diagram
$$
	\xymatrix{\cB \ar[d]_\rU \ar[r]^\rS & \cB \ar[d]^\rU\\
	\cA \ar[r]_\rC & \cA}
$$
commutes up to a natural
isomorphism $\Omega \colon \rC\rU \rightarrow \rU\rS$, then we
call $\rC$ an \emph{extension of} $\rS$ and $\rS$ a \emph{lift
of $\rC$ through the adjunction}.
\end{defn}

In general, any natural transformation $ \Omega \colon \rC\rU
\rightarrow \rU\rS$ uniquely determines a \emph{mate} $
\Lambda \colon \rF\rC \rightarrow \rS\rF$ that corresponds to
$$ \xymatrix{\rC \ar[r]^-{\rC \eta } & \rC\rU\rF
\ar[r]^{\Omega \rF} & \rU\rS\rF} $$ under the
adjunction \cite{MR2094071}.  The following theorem
constructs a canonical pair of distributive laws from
this mate of $ \Omega$:

\begin{thm}\label{arise}
Suppose that $\rS,\rC,$ and $\Omega$ are
as in Definition~\ref{lift}.  Then: \begin{enumerate}
\item The natural transformation
$$
	\xymatrix{\theta \colon
	\rB\rC=\rU\rF\rC \ar[r]^-{\rU \adj} & \rU\rS\rF
\ar[r]^-{\Omega^{-1}\rF} &\rC\rU\rF=\rC\rB}
$$
is a lax
endomorphism of the monad $\bB$.
\item The
natural transformation
$$ \xymatrix{\chi \colon
\rT\rS=\rF\rU\rS \ar[r]^-{\rF \Omega^{-1}} & \rF\rC\rU
\ar[r]^-{\adj \rU} & \rS\rF\rU=\rS\rT} $$ is a lax
endomorphism of the comonad $\bT$.  \item The lax
morphism $\theta$ is unique such that the following
diagram commutes:
$$
	\xymatrix{ \rU\rF\rC\rU
	\ar[rr]^-{\theta \rU} \ar[d]_-{\rU\rF \Omega}&
	&
	\rC\rU\rF\rU \ar[d]^-{\rC\rU \epsilon} \\
	\rU\rF\rU\rS
	\ar[r]_-{~\rU \epsilon \rS}& \rU\rS \ar[r]_-{\Omega^{-1}}
&	\rC\rU }
$$

\item The lax morphism $\chi$ is unique such that the
following diagram commutes:
$$
	\xymatrix{
	\rU\rS\ar[r]^-{\Omega^{-1}} \ar[d]_-{\eta \rU\rS} &
\rC\rU \ar[r]^-{C \eta \rU} &\rC\rU\rF\rU \ar[d]^-{\Omega
\rF\rU} \\ \rU\rF\rU\rS \ar[rr]_-{\rU \chi} &&
\rU\rS\rF\rU } $$ \item If $\rC$ is part of a comonad
$\bC=(\rC,\Delta^\rC,\varepsilon^\rC)$ and $\rS$ is
part of a comonad $\bS=(\rS,\Delta^\rS,\varepsilon^\rS)$
and $ \Omega $ is a
lax morphism of comonads, then $\theta$ is a mixed
distributive law and $ \chi $ is a comonad distributive
law.
\end{enumerate}
\end{thm}
\begin{proof}
To prove (1), observe that the unit compatibility
condition for $\theta$ is commutativity of the diagram
$$ \xymatrix{ \rU\rF\rC \ar[r]^-{\rU \Lambda} &
\rU\rS\rF \ar[d]^-{\Omega^{-1} \rF} \\ \rC \ar[u]^-{\eta
\rC} \ar[r]_-{\rC \eta} &\rC\rU\rF }
$$
This diagram commutes if and only if the same diagram post-composed
with $\Omega \rF$ commutes, which is exactly the fact
that $\Omega \rF \circ \rC\eta$ corresponds to ${\adj}$
under the adjunction. The multiplication compatibility
condition is given by commutativity of
$$
	\xymatrix{
	\rB\rB\rC\ar[d]_-{\rB \theta}
	\ar[r]^-{\mu\rC} & \rB\rC
	\ar[r]^-{\theta} &\rC\rB \\ \rB\rC\rB
	\ar[rr]_-{\theta{\rB}} &&
	\rC\rB\rB \ar[u]_-{\rC \mu} }
$$
which can be written as the outside of the diagram
$$
\xymatrix{ \rU\rF\rU\rF\rC \ar[r]^-{\rU \epsilon
\rF\rC}  \ar[d]_-{\rU\rF\rU \Lambda} & \rU\rF\rC
\ar[r]^-{\rU \Lambda} & \rU\rS\rF \ar[r]^-{\Omega^{-1}
\rF} &\rC\rU\rF \\ \rU\rF\rU\rS\rF\ar[d]_-{\rU\rF
\Omega^{-1} \rF} & & &  \\ \rU\rF\rC\rU\rF
\ar[rr]_-{\rU \Lambda \rU\rF} && \rU\rS\rF\rU\rF
\ar[uu]^-{\rU\rS \epsilon \rF} \ar[r]_-{\Omega^{-1} \rF
\rU \rF } &\rC\rU\rF\rU\rF \ar[uu]_-{\rC\rU \epsilon \rF}
}
$$
which will commute if both inner squares commute.
The right-hand square commutes by naturality of
$\Omega$. The left-hand square is obtained by applying
$\rU$ to the outside of the diagram $$ \xymatrix{
\rF\rU\rF\rC\ar[d]_-{\rF\rU \Lambda} \ar[r]^-{\epsilon
\rF\rC} & \rF\rC \ar[r]^-{\Lambda} & S\rF \\
\rF\rU\rS\rF\ar@{=}[r] \ar[d]_-{\rF \Omega^{-1} \rF}  &
\rF\rU\rS\rF \ar[ur]_-{\epsilon \rS\rF} &  \\
\rF\rC\rU\rF \ar[ur]_-{\rF \Omega \rF}
\ar[rr]_-{\Lambda \rU\rF} & & \rS\rF\rU\rF \ar[uu]_-{S
\epsilon \rF} } $$ which commutes: the upper shape
commutes by naturality of $\epsilon$, the left-hand
triangle clearly commutes, and the right-hand triangle
commutes since both morphisms are mapped to $\Omega$ by
the adjunction.

The proof for part (2) is similar to that of part (1).
For part (3), observe that the counit condition for
$\chi$ amounts to the commutativity of the diagram:
$$
	\xymatrix{ \rF\rU\rS \ar[r]^-{\rF\Omega^{-1}}
	\ar[d]_-{\epsilon \rS} &
	\rF\rC\rU \ar[d]^-{\Lambda \rU}
	\\
	\rS & \rS\rF\rU\ar[l]^-{\rS \epsilon} }
$$
If we
precompose this with $\rF \Omega^{-1}$ and then apply
$\rU$, we get the left-hand square of the diagram $$
\xymatrix@C=3em{ \rU\rF\rC\rU \ar[d]_-{\rU\rF \Omega}
\ar[r]^-{\rU \Lambda \rU} & \rU\rS\rF\rU
\ar[d]^-{\rU\rS \epsilon} \ar[r]^-{\Omega^{-1} \rF\rU}
&\rC\rU\rF\rU \ar[d]^-{\rC\rU \epsilon} \\ \rU\rF\rU\rS
\ar[r]_-{\rU \epsilon \rS} & \rU\rS \ar[r]_-{\Omega^{-1}}
&\rC\rU } $$ The right-hand square commutes by
naturality of $\Omega^{-1}$, so the outer square
commutes too, which is exactly the condition in part
(3).  Suppose that $\theta'$ is another lax morphism
which makes the diagram commute. Consider the diagram:
$$
\xymatrix{ \rU\rF\rC \ar[rr]^-{\theta'}
\ar[d]_-{\rU\rF\rC \eta} &&\rC\rU\rF \ar[d]^-{\rC\rU\rF
\eta}  \ar@{=}@/^4pc/[dd]\\ \rU\rF\rC\rU\rF
\ar[d]_-{\rU\rF \Omega \rF} \ar[rr]^-{\theta' \rU\rF}
&&\rC\rU\rF\rU\rF \ar[d]^-{\rC\rU \epsilon \rF} \\
\rU\rF\rU\rS\rF \ar[r]_-{\rU \epsilon \rS\rF} & \rU\rS\rF
\ar[r]_-{\Omega^{-1} \rF} & \rC\rU\rF }
$$
The
rightmost shape commutes by one of the triangle
identities for the adjunction, the bottom square
commutes by hypothesis, and the upper square commutes
by naturality of $\theta'$. Therefore, the outer
diagram commutes which says exactly that $$ \theta' =
\Omega^{-1} \rF \circ \rU(\epsilon \rS\rF \circ \rF
\Omega \rF \circ \rF\rC \eta) = \Omega^{-1} \rF \circ
\rU \Lambda = \theta.  $$

For part (4), the displayed diagram commutes for
similar reasons to the diagram in part (3). Let $\chi'$
be another lax morphism such that the diagram commutes.
Going round the diagram clockwise shows that $\chi$ and
$\chi'$ are mapped to the same morphism under the
adjunction, so $\chi = \chi'$.

For part (5), we will show that $\theta$ is a mixed
distributive law, and remark that the proof that $\chi$
is a comonad distributive law is similar. Consider the
following diagram:
$$
\xymatrix@R=4em{ & \rU\rF &\\
\rU\rF\rC \ar[ur]^-{\rU\rF \epsilon^\rC}
\ar[r]_-{\theta} &\rC\rU\rF \ar[u]_-(.35){\epsilon^\rC
\rU\rF} \ar[r]_-{\Omega \rF} & \rU\rS\rF \ar[ul]_-{\rU
\epsilon^\rS \rF} }
$$
The left hand triangle, which is the
counit compatibility condition for $\theta$, will
commute if the right-hand and outer triangle commute.
The right-hand triangle commutes because $\Omega$ is
lax by hypothesis. The outer triangle is just $\rU$
applied to the diagram
$$
	\xymatrix{ \rF\rC
	\ar[d]_-{\Lambda} \ar[r]^-{\rF \epsilon^\rC} & \rF \\
	\rS\rF \ar[ur]_-{\epsilon ^\rS \rF} }
$$
This commutes since the mate of a lax morphism is always colax
\cite[p180]{MR2094071}. By a similar argument, $\theta$ is
compatible with the comultiplication.  \end{proof}

\begin{defn}
A comonad distributive law $ \chi $ as in
Theorem~\ref{arise} is said to \emph{arise from the
adjunction $\rF \dashv \rU$}.
\end{defn}

\begin{exa}\label{banal}
A trivial example
which will nevertheless play a r\^ ole below is
the case where $\rC=\rB$, $\rS=\rT$, and $ \Omega =
{\idty} $. In this case,
$ \chi $ and $ \theta$ are given by
\begin{align*}
& \xymatrix{\rT\rT=\rF\rU\rF\rU \ar[r]^-{\varepsilon
	\rF\rU} & \rF\rU \ar[r]^-{\rF \eta \rU} &
\rF\rU\rF\rU=\rT\rT,}\\
&
	\xymatrix{\rB\rB=\rU\rF\rU\rF \ar[r]^-{\rU \varepsilon
\rF} & \rU\rF \ar[r]^-{\rU\rF \eta } &
\rU\rF\rU\rF=\rB\rB.}
\end{align*}
\end{exa}

\subsection{The Eilenberg-Moore and the coKleisli
cases}\label{exceptional}
Functors do not
necessarily lift respectively extend through an
adjunction (for example, the functor on $\cSet$
which assigns the empty set to each set does not lift
to ${k\cMod}$), and if they do, they may not do so
uniquely.  Theorem~\ref{arise} says only that once a
lift respectively extension is chosen, there
is a unique compatible pair of lax endomorphisms $
\theta $ and $\chi$.

One extremal situation in which specifying a lax
endomorphism $\theta \colon \rC \bB \rightarrow
\bB \rC$ uniquely determines a lift $\rS$ of $\rC$ is
when $\cB$ is the Eilenberg-Moore category $\cA^\bB$.
In this case, $\rS$ is defined on objects
$(X,\alpha)$ by $\rS(X,\alpha)=(\rC X,\rC \alpha \circ \theta
X)$. Using Theorem~\ref{arise} (with $\Omega =
{\idty} $), one recovers $ \theta $, see,
\emph{e.g.},~\cite{MR2614998,MR0390018}.

Dually, one can take
$\cA$ to be the coKleisli category $\cB_\bT$ in
which case a lax endomorphism $ \chi $ yields an extension
$\rC$ of a functor $\rS$. This means that every comonad
distributive law and every mixed distributive law
arises from an adjunction.

\subsection{The comparison functor is a
$1$-cell}\label{twocatone}

Let $\rF \dashv \rU$ be an adjunction and let $\bS$
be the lift of a comonad $\bC$ through the
adjunction via $\Omega$ as in Section~\ref{adj}.
Suppose we have a $1$-cell
$$
\xymatrix@C=1.5em{ & \cA
\ar@{}[r]^{\bB}="a" \ar@{}[l]_{\bC}="b"
\ar[d]^-{(\Sigma, \sigma, \gamma)}& \\ & \cD
\ar@{}[r]_{\bA}="c" \ar@{}[l]^{\bD}="d" &
\ar @{.}@/_/ "a";"b" |{\theta} \ar @{.}@/^/ "c";"d"
|{\psi} } $$ in the $2$-category
$\cMix$. Let
us denote with tildes the lifts of $\bA,\bD$, and $\psi$
to the Eilenberg-Moore category $\cD^\bA$
outlined in Section~\ref{exceptional}. This gives rise
to a $1$-cell
$$
	\xymatrix@C=1.5em{ & \cB
	\ar@{}[r]^{\bT}="a" \ar@{}[l]_{\bS}="b"
	\ar[d]^-{(\tilde\Sigma, \tilde\sigma, \tilde\gamma)}&
\\ & \cD^\bA \ar@{}[r]_{\tilde{\bA}}="c"
\ar@{}[l]^{\tilde{\bD}}="d" & \ar @{.}@/_/
"a";"b" |{\chi} \ar @{.}@/^/ "c";"d" |{\tilde\psi} }
$$
in $\cDist$,
where $\tilde\Sigma$ is defined on objects by
$$
\tilde\Sigma X = \left( \Sigma \rU X,\xymatrix{ \rA
\Sigma \rU X \ar[r]^-{\sigma \rU X} & \Sigma \rB \rU X
= \Sigma \rU \rF \rU X \ar[r]^-{\Sigma \rU \epsilon X}
& \Sigma \rU X} \right) $$ and on morphisms by
$\tilde\Sigma f = \Sigma \rU f$. The lax morphism
$\tilde\sigma$ is defined by $$ \xymatrix{ \rA \Sigma \rU
X \ar[r]^-{\sigma \rU X} & \Sigma \rB \rU X = \Sigma
\rU \rT X}
$$
and the colax morphism $\tilde\gamma$ is
defined by
$$
	\xymatrix{ \Sigma \rU \rS X \ar[r]^-{\Sigma
\Omega^{-1} X} & \Sigma \rC \rU X \ar[r]^-{\gamma \rU X}
& \rD \Sigma \rU X}
$$

In the case that $\cA= \cD$, $\bB = \bA$,
$\bC = \bD$, $\psi = \theta$ and $(\Sigma,
\sigma, \gamma) = ({\idty}, {\idty},
{\idty})$ is the trivial $1$-cell, we get that
$\tilde\Sigma$ is the \emph{comparison functor}
$\cB
\rightarrow \cA^\bB = \cD^\bA$.

\subsection{Interpretation as a
$2$-functor}\label{twocattwo}
Consider the
case that $\cB = \cA^\bB$, $\bT =
\tilde{\bB}$, $\bS = \tilde{\bC}$,
and $\chi = \tilde\theta$. Since any $2$-cell $\alpha
\colon \Sigma \rightarrow \Sigma'$ lifts to a natural
transformation $\tilde\alpha \colon \tilde\Sigma \rightarrow
\tilde \Sigma'$, we can encode the above construction
as the action of a $2$-functor:
\begin{prop}
The assignment
$$ \xymatrix@C=1.5em{ & \cA
\ar@{}[r]^{\bB}="a" \ar@{}[l]_{\bC}="b"
\ar[d]^-{(\Sigma, \sigma, \gamma)}& &
\ar@{}[d]|{\lmapsto} &  \cA^\bB
\ar@{}[r]^{\tilde{\bB}}="e"
\ar@{}[l]_{\tilde{\bC}}="f"
\ar[d]^-{(\tilde\Sigma, \tilde\sigma, \tilde\gamma)}&
\\ & \cD \ar@{}[r]_{\bA}="c" \ar@{}[l]^{\bD}
="d" & & &   \cD^\bA \ar@{}[r]_{\tilde{\bA}}="g" \ar@{}[l]^{\tilde{\bD}}="h" & \ar
@{.}@/_/ "a";"b" |{\theta} \ar @{.}@/^/ "c";"d" |{\psi}
\ar @{.}@/_/ "e";"f" |{\tilde\theta} \ar @{.}@/^/
"g";"h" |{\tilde\psi} } \quad\quad\quad\quad
\xymatrix@C=0.5em{ (\Sigma, \sigma, \gamma)
\ar@{=>}[d]^-{\alpha} & \ar@{}[d]|\lmapsto &
(\tilde\Sigma, \tilde\sigma, \tilde\gamma)
\ar@{=>}[d]^-{\tilde\alpha}\\ (\Sigma', \sigma',
\gamma') & & (\tilde\Sigma',
\tilde\sigma',\tilde\gamma') } $$ is a $2$-functor $i
\colon \cMix
\rightarrow \cDist$.
\end{prop}

Analogously, we obtain a $2$-functor $ j \colon
\cDist \rightarrow
\cMix$ by taking extensions to coKleisli
categories.  It is
those distributive laws in the image of the $2$-functor
$i$ that are the main object of study in this paper.

\subsection{The Galois map}\label{galoismapsct}
Theorem~\ref{arise} yields comonad distributive laws
from lifts through an adjunction, and different lifts
produce different distributive laws. Here we describe
how these are related in terms of suitable
generalisations of the Galois map from the theory of
Hopf algebras.

\begin{defn}
If $\rS,\rV \colon \cB \rightarrow \cB$
are lifts of $\rC \colon \cA
\rightarrow \cA$
through $\rF \dashv \rU$ with
isomorphisms $\Omega \colon\rC\rU \rightarrow \rU\rS$ and
$\Phi \colon\rC\rU \rightarrow \rU\rV$, we define a natural
isomorphism
$$
	\Gamma^{\rS,\rV} \colon \cB(\rF-, \rS-) \rightarrow
\cB(\rF-, \rV-)
$$
of functors $\cA^\op \times \cB \rightarrow \cSet$
on components by the composition
$$
	\xymatrix{ \cB (\rF X, \rS Y) \ar[r]
	&
	\cA(X, \rU\rS Y) \ar[r] &
	\cA(X, \rU\rV Y) \ar[r]
	& \cB
	(\rF X, \rV Y), }
$$
where the middle map is induced by
$ \Phi_Y \circ \Omega^{-1}_Y \colon \rU\rS Y \rightarrow
\rU\rV Y$ and the outer ones are induced by the adjunction
$\rF \dashv \rU$.
We call $ \Gamma^{\rS,\rV}$ the \emph{Galois
map} of the pair $(\rS,\rV)$.  \end{defn}

The following properties
are easy consequences of the definition:

\begin{prop}\label{sunshines}
Let $\rS$ and $\rV$ be two lifts
of an endofunctor $\rC$ through an adjunction $\rF \dashv
\rU$. Then:
\begin{enumerate}
\item The inverse of
$\Gamma^{\rS,\rV}$ is given by $\Gamma^{\rV,\rS}$.
\item The Galois map $\Gamma^{\rS,\rV}$ maps a morphism $f
\colon \rF X \rightarrow \rS Y$ to
$$
	\xymatrix{ \rF X
\ar[r]^-{\rF \eta X} & \rF\rU\rF X \ar[r]^-{\rF\rU f} &
\rF\rU\rS Y \ar[rr]^{\rF(\Phi_Y \circ \Omega^{-1}_Y)} &
& \rF\rU\rV Y \ar[r]^-{\epsilon \rV Y} & \rV Y. }
$$
\item If $\chi^\rS$ and $ \chi^\rV$ denote the lax morphisms
determined by the two lifts, then
$$
	\Gamma^{\rS,\rV}
	(\chi^\rS ) = \chi^\rV.
$$
\end{enumerate}

\end{prop}

So, in the applications of Theorem~\ref{arise}, all
distributive laws obtained from different lifts of a
given comonad through an adjunction are obtained from
each other by application of the appropriate Galois
map.

The Galois map also relates different lifts of $\rB$
itself: recall the trivial Example~\ref{banal} of
Theorem~\ref{arise}, where $\rC=\rB$ and $\rS=\rT$, and let
$\rV$ be any other lift of $\rB$ through the
adjunction. By taking $X$ to be $\rU Y$ for an object
$Y$ of $\cB$, one obtains a Galois map $ \Gamma^{\rT,\rV}
\colon \cB (\rT-,\rT-) \rightarrow \cB(\rT-,\rV-) $ that we can
evaluate on $ {\idty} \colon \rT Y \rightarrow \rT Y$, which
produces a natural transformation $\rT \rightarrow \rV$
that we denote by slight abuse of notation by $
\Gamma^{\rT,\rV}$ as well.

Adapting \cite[Definition 1.3]{MR2651345}, we define:
\begin{defn} We say that $\rF$ is \emph{$\rV$-Galois}
if $$ \xymatrix{ \Gamma^{\rT,\rV} \colon \rT=\rF\rU
\ar[r]^-{\rF \eta \rU} & \rF\rU\rF\rU = \rF\rU\rT
\ar[r]^-{\rF \Phi} & \rF\rU\rV  \ar[r]^-{\epsilon \rV}
& \rV } $$ is an isomorphism.  \end{defn}

The following proposition provides the connection to
Hopf algebra theory:

\begin{prop}\label{wisga}
If $\rF$ is $\rV$-Galois and
$ \theta \colon \rB\rB \rightarrow \rB\rB$ is the lax
morphism arising from the lift $\rV$ of $\rB$, then the
natural transformation
$$
	\xymatrix{ \beta \colon \rB\rB
	\ar[r]^-{\rB \eta \rB} & \rB\rB\rB \ar[r]^-{\theta \rB}
	& \rB\rB\rB \ar[r]^-{\rB \mu} & \rB\rB  }
$$
is an
isomorphism.
\end{prop}
\begin{proof}
If $\rF$ is
$\rV$-Galois, then $ \rU\Gamma^{\rT,\rV}\rF$ is an
isomorphism
$$
	\xymatrix{ \rU\rT\rF=\rU\rF\rU\rF
\ar[rr]^-{\rU\rF \eta \rU\rF} & & \rU\rF\rU\rF\rU\rF =
\rU\rF\rU\rT \rF \ar[r]^-{\rU\rF \Phi \rF} &
\rU\rF\rU\rV \rF \ar[r]^-{\rU\epsilon \rV\rF} &
\rU\rV\rF. }
$$
Let now $ \chi \colon \rT\rV \rightarrow
\rV\rT$ be the lax morphism corresponding to $ \theta $
as in Theorem~\ref{arise}.  Inserting $\varepsilon
\rV=(\rV \varepsilon) \circ \chi$ and $\rU \chi \circ
\rU\rF \Phi=\Phi \rF\rU \circ \theta \rU$ and
$\rB=\rU\rF$, the isomorphism becomes $$ \xymatrix{
\rU\rT\rF=\rB\rB \ar[r]^-{\rB \eta \rB} & \rB\rB\rB
\ar[r]^-{\theta \rB} & \rB\rB\rB=\rB\rU\rF\rU\rF
\ar[r]^-{\Phi \rF\rU\rF} & \rU\rV\rF\rU\rF
\ar[r]^-{\rU\rV\epsilon \rF} & \rU\rV\rF } $$ Finally,
we have by construction $\rU \varepsilon \rF=\mu $, and
using the naturality of $\Phi$ this gives
$\rU\rV\varepsilon   \rF \circ \Phi \rF\rU\rF= \Phi \rF
\circ \rB\rU\varepsilon \rF$.  Hence composing the
above isomorphism with $\Phi^{-1} \rF$ gives $ \beta $.
\end{proof}

It is this associated map $ \beta $ that is used to
distinguish Hopf algebras amongst bialgebras, see
Section~\ref{wisbimonad} below.

\section{Coefficients}\label{coeffsec}
\subsection{Coalgebras over distributive laws}\label{chicoalgs} Let
$\bT = \left(\rT, \Delta^\rT, \epsilon^\rT\right)$
and $\bS = \left(\rS, \Delta^\rS,
\epsilon^\rS\right)$ be comonads on a category $\cB$, and let
$\chi \colon \bT \bS\rightarrow \bS \bT$
be a distributive law.  We now discuss $ \chi
$-coalgebras, which serve as
coefficients in the homological constructions in the
next section.

\begin{defn}\label{coalgdef}
A \emph{right
$\chi$-coalgebra} is a triple $(\rM,\cX, \rho)$,
where
$\rM \colon \cX
\rightarrow \cB$ is a functor and $\rho \colon \rT\rM \rightarrow \rS\rM$
is a natural transformation such that the diagrams
$$
\xymatrix{ \rT\rM \ar[r]^-{\Delta^\rT\rM} \ar[d]_-{\rho} &
\rT\rT\rM \ar[r]^{\rT\rho} &\rT\rS\rM \ar[d]^-{\chi \rM} \\
\rS\rM \ar[r]_-{\Delta^\rS \rM} & 	\rS\rS\rM & \rS\rT\rM
\ar[l]^-{\rS\rho} } \quad\quad\quad
\xymatrix{ & \rT\rM
\ar[dl]_-{\epsilon^\rT\rM} \ar[d]^-{\rho}\\ \rM & \rS\rM
\ar[l]^-{\epsilon^\rS \rM} } $$ commute.  Dually, we
define \emph{left $\chi$-coalgebras} $(\rN, \cY,
\lambda)$.
\end{defn}

The following characterises right $ \chi $-coalgebras
in the setting of Theorem~\ref{arise}.

\begin{prop}\label{chicoalgprop}
In the situation of
Theorem~\ref{arise}, let $\rM \colon \cX \rightarrow \cB$
be a functor.
\begin{enumerate}
\item Right
$\chi$-coalgebra structures $\rho$ on $\rM$ correspond to
$\bC$-coalgebra structures $\nabla$ on the
functor $\rU\rM \colon \cX \rightarrow \cA$.
\item Let
$\rS$ and $\rV$ be two lifts of the functor $\rC$ through the
adjunction, and let $\chi^\rS$ and $\chi^\rV$ denote the
comonad distributive laws determined by the lifts
$\rS$ and $\rV$ respectively. Then the Galois map
$\Gamma^{\rS,\rV}$ maps right $\chi^\rS$-coalgebra
structures $\rho^\rS$ on $\rM$ bijectively to right
$\chi^\rV$-coalgebra structures $\rho^\rV$ on $\rM$.
\end{enumerate}
\end{prop}

\begin{proof} For part (1),
right $\chi$-coalgebra structures $\rho \colon \rF\rU\rM
\rightarrow \rS\rM$ are mapped under the adjunction to $\nabla
\colon \rU\rM \rightarrow \rU\rS\rM \cong\rC\rU\rM$. Part (2)
follows immediately since the Galois map is
the composition of the adjunction isomorphisms and
$\Phi \circ \Omega^{-1}$.  \end{proof}

\subsection{{Entwined} $ \chi $-coalgebras}
In the remainder of this section, we discuss a class of
coefficients that
lead
to contractible simplicial objects, see
Proposition~\ref{trivcontract} below.  In the Hopf
algebroid setting, these are the Hopf (or entwined)
modules as studied in \cite{MR3020336,MR1604340}.
First, we recall:

\begin{defn}
A \emph{$\bT$-coalgebra} is a triple
$(\rM, \cX, \nabla)$, where $\rM \colon \cX
\rightarrow \cB$ is a functor and $\nabla \colon \rM \rightarrow \rT\rM$ is
a natural transformation such that the diagrams
$$
	\xymatrix{\rM \ar[r]^-{\nabla}
	\ar[d]_-\nabla & \rT\rM
	\ar[d]^-{\Delta^\rT\rM} \\
	\rT\rM \ar[r]_-{\rT \nabla} &\rT\rT\rM
}
\quad\quad\quad
\xymatrix{ \rM \ar@{=}[dr]
\ar[r]^-\nabla & \rT\rM \ar[d]^-{\epsilon^\rT\rM} \\ &
\rM }
$$
commute.
\end{defn}

Dually, one defines $\bT$-opcoalgebras
$(\rN, \cY, \nabla)$ where $\nabla \colon \rN
\rightarrow \rN\rT$, as well as
algebras and opalgebras involving monads.  Note that
$\bT$-coalgebras can be equivalently viewed as
$1$-cells from respectively to the trivial distributive
law:

\begin{prop}\label{triv1cell}
Given an $\bS$-coalgebra
$(\rM, \cX, \nabla^\rS)$ and a
$\bT$-opcoalgebra $(\rN, \cY, \nabla^\rT)$,
there is a pair of $1$-cells
$$
	\xymatrix@C=1.5em{ &
	\cX \ar@{}[r]^{{\idty}}="a"
	\ar@{}[l]_{{\idty}}="b"
	\ar[d]^-{\left(\rM,~\epsilon^\rT\rM,~\nabla^\rS\right)}&
\\ & \cB \ar@{}[r]_{\bT}="c" \ar@{}[l]^{\bS}
="d" & \ar @{.}@/_/ "a";"b" |{{\idty}} \ar
@{.}@/^/ "c";"d" |{\chi} } \quad\quad\quad
\xymatrix@C=1.5em{ & \cB \ar@{}[r]^{\bT}="a"
\ar@{}[l]_{\bS}="b"
\ar[d]^-{\left(\rN,~\nabla^\rT,~\rN \epsilon^\rS\right)}&
\\ & \cY \ar@{}[r]_{{\idty}}="c"
\ar@{}[l]^{{\idty}}="d" & \ar @{.}@/_/ "a";"b"
|{\chi} \ar @{.}@/^/ "c";"d" |{{\idty}} } $$ and
all $1$-cells ${\idty} \rightarrow \chi $ respectively $\chi
\rightarrow {\idty}$ are of this form.  \end{prop}

Furthermore, these $1$-cells can also be viewed as $ \chi
$-coalgebras:

\begin{prop}\label{triv}
Let $\chi \colon \bT\bS \rightarrow \bS\bT$ be a comonad
distributive law. Then:
\begin{enumerate}
\item Any $\bS$-coalgebra $(\rM,
\cX, \nabla^\rS)$ defines a right
$\chi$-coalgebra $(\rM, \cX, \epsilon^\rT
\nabla^\rS)$.  \item Any $\bT$-opcoalgebra $(\rN,
\cY, \nabla^\rT)$ defines a left
$\chi$-coalgebra $(\rN, \cY, \nabla^\rT
\epsilon^\rS)$.
\end{enumerate}
\end{prop}

\begin{defn}
If a $\chi$-coalgebra arises from an
(op)coalgebra as in Proposition~\ref{triv}, then we
call the $\chi$-coalgebra \emph{{entwined}}.
\end{defn}

Note, however, that there is no obvious way to associate
a $1$-cell in $\cDist$ to an arbitrary right or left $\chi$-coalgebra.

\subsection{Entwined algebras}
Finally, we describe how
entwined $ \chi $-coalgebras are in some sense lifts of
entwined algebras; throughout,
$\theta \colon \bB\bC \rightarrow \bC\bB$ is a mixed
distributive law between a monad $\bB$ and a
comonad $\bC$ on a category $\cA$.

\begin{defn}
Let $\rM \colon
\cX \rightarrow \cA$ be a functor which has a
$\bB$-algebra structure $\beta \colon \rB \rM \rightarrow \rM$ and
a $\bC$-coalgebra structure $\nabla \colon \rM \rightarrow
\rC\rM$. We say that the quadruple $(\rM, \cX, \beta,
\nabla)$ is an \emph{entwined algebra with respect to
$\theta$} if the diagram
\begin{equation}\label{entwinedcondition}		
\begin{array}{c}
\xymatrix{ \rB \rM
\ar[d]_-{\rB \nabla} \ar[r]^-{\beta} & \rM
\ar[r]^-{\nabla} & \rC\rM\\ \rB\rC \rM \ar[rr]_{\theta
\rM} & &
\rC\rB \rM \ar[u]_-{\rC \beta} }
\end{array}
\end{equation}
commutes.
\end{defn}

Dually we define an entwined opalgebra structure
on a
functor $\rN \colon \cA \rightarrow \cY$
for a distributive law $\mathbb{CB} \rightarrow \mathbb{BC}$.

The following proposition explains the relation
between entwined algebras and
entwined right $ \chi $-coalgebras for distributive
laws $ \chi $ arising from an adjunction:

\begin{prop}
In the situation of
Theorem~\ref{arise}, let $\rM \colon \cX \rightarrow \cB$
be a functor and let $\nabla \colon \rM \rightarrow
\rS\rM$ be a natural transformation.
\begin{enumerate}
\item If $\nabla$ is an
$\bS$-coalgebra structure,
then the structure morphisms
$$
\xymatrix{ \rB\rU\rM = \rU\rF\rU\rM
\ar[r]^-{\rU \epsilon \rM} & \rU\rM}, \qquad \xymatrix{
\rU\rM \ar[r]^-{\rU \nabla} & \rU\rS\rM
\ar[r]^-{\Omega^{-1}} &\rC\rU\rM}
$$
turn $\rU\rM$ into an entwined
algebra with respect to $\theta$.
\item If $\cB =
\cA^\bB$, then the converse of (1) holds.
\end{enumerate} \end{prop} \begin{proof} For part (1),
the morphism $\rB\rU\rM\rightarrow \rU\rM$ is the
$\bB$-algebra structure on $\rM$ given by the
comparison functor, and the morphism
$\rU\rM \rightarrow\rC\rU\rM$ is the $\bC$-coalgebra
structure given by Proposition~\ref{chicoalgprop}.
The commutativity of (\ref{entwinedcondition}) follows
by applying the functor
$\rU$ to the Yang-Baxter condition for the $1$-cell
$\left(\rM,~\epsilon^\rT\rM,~\nabla^\rS\right)$ of
Proposition~\ref{triv1cell}. For part (2),
condition (\ref{entwinedcondition}) means exactly that
the $\bC$-coalgebra
structure defines a morphism in $\cA^\bB$, and
hence lifts to an $\bS$-coalgebra structure.
\end{proof}

Dually, entwined opalgebra structures on a $\mathbb B$-opalgebra $(\rN, \mathcal Z, \omega)$ are related to
left $\chi$-coalgebras if
the codomain $\cY$ of $\rN$ is a
category with coequalisers. First, we define a functor
$\rN_\bB \colon \cA^\bB \rightarrow \cY$
that takes a $\bB$-algebra morphism $f
\colon (X, \alpha) \rightarrow (Y, \beta)$ to
$\rN_\bB(f)$ defined using coequalisers:
$$
\xymatrix@C=5em{ \rN\rB X \ar[d]_-{\rN \rB f}
\ar@<+.5ex>[r]^-{\omega_X} \ar@<-.5ex>[r]_-{\rN \alpha} &
\ar[d]_-{\rN f} \rN X \ar@{->>}[r]^-{q_{(X,\alpha)}} &
\rN_\bB (X, \alpha) \ar@{.>}[d]^-{\rN_\bB(f)} \\ \rN\rB
Y  \ar@<+.5ex>[r]^-{\omega_Y}
\ar@<-.5ex>[r]_-{\rN \beta} & \rN Y \ar@{->>}[r]_-{q_{(Y,
\beta)}} & \rN_\bB(Y, \beta) }
$$
Thus $ \rN_\bB$ generalises the functor
$- \otimes_B N$ defined by a left module $N$
over a ring $B$ on the category of right $B$-modules.

Suppose that
$\theta$ is invertible, and that $\rN$ admits the
structure of an entwined $\theta^{-1}$-opalgebra, with
coalgebra structure $\nabla \colon \rN \rightarrow
\rC\rN$. There
are two commutative diagrams:
$$
	\xymatrix{ \rN\rB
	X\ar[d]_-{\nabla_{\rB X}} \ar[rr]^-{\omega_X} &&
	\rN X
	\ar[dd]^-{\nabla_X} \\ \rN\rC\rB X
	\ar[d]_-{\rN \theta^{-1}_X} & & \\ \rN\rB\rC X
	\ar[rr]_-{\omega_{\rC X}} & & \rN\rC X
	}\quad\quad\quad
	\xymatrix{ \rN\rB X
	\ar[d]_-{\nabla_{\rB X}}
	\ar[rr]^-{\rN\alpha} & &
	\rN X \ar[dd]^-{\nabla_X} \\
	\rN\rC\rB X
	\ar[d]_-{\rN \theta^{-1}_X} & & \\
	\rN\rB\rC X
	\ar[r]_-{\rN \theta_X} &
	\rN\rC\rB X
	\ar[r]_-{\rN\rC \alpha} &
	\rN\rC X }
$$
Hence, using coequalisers,
$\nabla$ extends to a natural transformation $\tilde
\nabla \colon \rN_\bB \rightarrow \rN_\bB \tilde
\rC$,  and in fact it gives $\rN_\bB$ the structure
of a $\tilde{\bC}$-opcoalgebra. Since
$\tilde{\theta}^{-1} \colon \tilde{\bC}
\tilde{\bB} \rightarrow \tilde{\bB}\tilde{\bC}$
is a comonad distributive law on
$\cA^\bB$, Proposition~\ref{triv} gives us the
following:
\begin{prop}
The triple $(\rN_\bB,
\cY, \tilde{\nabla} \epsilon)$ is an {entwined} left
$\tilde\theta^{-1}$-coalgebra.  \end{prop}

\section{Duplicial objects}\label{duplobjsec}
\subsection{The bar and opbar resolutions} Let $\bT
= (\rT, \Delta,\epsilon)$ be a comonad on a category
$\cB$, and let $\rM \colon \cX \rightarrow \cB$ be a
functor.
\begin{defn}
The \emph{bar resolution of }$\rM$
is the simplicial functor
$ \rBB (\bT, \rM) \colon \cX \rightarrow
\cB$ defined by
$$
	\rBB (\bT, \rM)_n =
	\rT^{n+1}\rM, \qquad
	d_i =\rT^i \epsilon
	\rT^{n-i} \rM, \qquad
	s_j =\rT^j \Delta\rT^{n-j}
	\rM,
$$
where the face and degeneracy maps
above are given in degree $n$. The \emph{opbar
resolution of }$\rM$,
denoted
$\rBB^\op(\bT, \rM)$, is the simplicial
functor obtained by taking the op\-sim\-pli\-cial
sim\-pli\-cial
functor of $\rBB (\bT, \rM)$. Explicitly:
$$
 	\rBB^\op(\bT, \rM)_n =
\rT^{n+1}\rM,\qquad d_i =\rT^{n-i} \epsilon\rT^{i} \rM,\qquad
	s_j =\rT^{n-j} \Delta\rT^{j} \rM.
$$
\end{defn}
Given any functor $\rN \colon \cB \rightarrow
\cY$, we compose it with the above simplicial
functors to obtain new simplicial functors that we
denote by
$$
	\rCC_\bT(\rN,\rM):=\rN\rBB (\bT,M), \qquad
\rCC^\op_\bT(\rN,\rM):=\rN\rBB ^\op(\bT, \rM).
$$

\subsection{Duplicial objects} Duplicial objects were
defined by Dwyer and Kan \cite{MR826872} as a mild
generalisation of Connes' cyclic objects
\cite{Con:CCEFE}:

\begin{defn}
A \emph{duplicial object} is a simplicial
object $(C,d_i,s_j)$ together with additional morphisms
$t \colon C_n \rightarrow C_n$ satisfying
\[
	d_i t =
	\begin{cases}
	t d_{i-1}, & 1 \leq i \leq n, \\
	d_n,
	&
	i=0, \end{cases} \qquad
	s_{j}t = \begin{cases} t
	s_{j-1}, & 1 \leq j \leq n, \\
	t^2 s_n, & j = 0.
	\end{cases}
\]
A duplicial object is \emph{cyclic} if $
T:=t^{n+1} = {\idty}.  $
\end{defn}

Equivalently, a duplicial object is a simplicial object
which has in each degree an \emph{extra degeneracy}
$ s_{-1} \colon C_n
\rightarrow C_{n+1}$. This corresponds
to $t$ via
$$
	s_{-1} := t s_n,\quad t=d_{n+1} s_{-1}.
$$
This turns a duplicial object also into a
cosimplicial object, and hence a duplicial object $C$
in an additive category carries a boundary
and a coboundary map
$$
	b:=\sum_{i=0}^n
	(-1)^i d_i,\quad
	s:=\sum_{j=-1}^n (-1)^j s_j.
$$
Dwyer
and Kan called such chain and cochain complexes
\emph{duchain complexes} and showed that the normalised
chain complex functor yields an equivalence between
duplicial objects and duchain complexes in an abelian
category, thus extending the classical Dold-Kan
correspondence between simplicial objects and chain
complexes.

If $f_n \in \mathbb{Z} [x]$ is given by $1-x f_n(x) =
(1-x)^{n+1}$ and $B:= s f_n(bs)$, then one has $$
B^2=0,\quad bB+Bb={\idty}-T, $$ and in this way
cyclic objects give rise to mixed complexes $(C,b,B)$
in the sense of \cite{MR883882} that can be used to
define \emph{cyclic homology}.

\subsection{The B\"ohm-\c Stefan construction}
\label{evidenziatore1}
Let
$(\cB, \chi, \bT, \bS)$ be a
$0$-cell in $\cDist$, and let $(\rM, \cX, \rho)$ and $(\rN, \cY, \lambda)$
be right and left $\chi$-coalgebras
respectively.  By abuse of notation, we let $\chi^n$
denote both natural transformations
$\rT^n \rS \rightarrow \rS\rT^n$ and
$\rT \rS^n \rightarrow \rS^n \rT$ obtained by repeated application of
$\chi$ (up to horizontal composition of identities),
where $\chi^0 = {\idty}$.  We furthermore define
natural transformations
$$
	t^\bT_n \colon
	\rCC_\bT(\rN,\rM)_n \rightarrow \rCC_\bT(\rN,\rM)_n,
	\quad t^\bS_n \colon
	\rCC_\bS^\op(\rN,\rM)_n \rightarrow \rCC_\bS^\op(\rN,\rM)_n
$$
by the diagrams
$$
	\xymatrix@C=3.5em{ \rN \rT^{n} \rS\rM
	\ar[r]^-{\rN  \chi^n \rM} &
	\rN \rS\rT^n \rM
	\ar[d]^-{\lambda\rT^n \rM}
	\\
	\rN \rT^{n+1}\rM
	\ar[u]^-{\rN \rT^n \rho}
	\ar@{.>}[r]_-{t^\bT_n}  & \rN \rT^{n+1} \rM }
	\quad\quad\quad
	\xymatrix@C=3.5em{ \rN\rT\rS^n\rM
	\ar[r]^-{\rN \chi^n \rM} & \rN\rS^n \rT \rM
	\ar[d]^-{\rN \rS^n \rho}
	\\ \rN\rS^{n+1} \rM \ar[u]^-{\lambda \rS^n \rM}
	\ar@{.>}[r]_-{t_n^\bS} &\rN\rS^{n+1} \rM }
$$

\begin{thm}\label{dup}
The simplicial functors
$
	\rCC_\bT(\rN,\rM)
$
and
$
	\rCC^\op_\bS(\rN, \rM)
$
become duplicial functors with duplicial
operators given by $t^\bT$ respectively
$t^\bS$.
\end{thm}
\begin{proof}
The first operator being duplicial is exactly the case considered
in \cite{MR2415479}, and the second follows from a slight
modification of their proof.  \end{proof}

\subsection{Cyclicity}
\label{evidenziatore2}
For each $n \ge 0$, we define a morphism
$R_n \colon \rN\rT^{n+1} \rM \rightarrow \rN\rS^{n+1} \rM$
in the following way. For
each $0 \le i \le n$, let $r_{i,n}$ denote the morphism
$$
	\xymatrix{ \rN \rS^i \rT^{n+1-i} \rM \ar[rr]^-{ \rN \rS^i
\rT^{n-i} \rho} && \rN \rS^i\rT^{n-i}\rS \rM \ar[rr]^-{N \rS^i
\chi^{n-i} \rM} & & \rN \rS^{i+1}\rT^{n-i} \rM.}
$$
Then set
$$
	R_n := r_{n,n} \circ \cdots \circ r_{0,n}.
$$
Similarly,
we can define a morphism $L_n \colon  \rN\rS^{n+1} \rM \rightarrow
\rN\rT^{n+1} \rM$ whose definition involves the left
$\chi$-coalgebra structure $\lambda$ on $\rN$.

\begin{prop}\label{cyc}
The above construction defines
two morphisms
\begin{align*}
	\xymatrix{\rCC_\bT(\rN,\rM) \ar[r]^-{R }  &
\rCC_\bS^\op(\rN,\rM)}, \qquad
	\xymatrix{\rCC_\bS^\op(\rN,\rM)
	\ar[r]^-{L } & \rCC_\bT(\rN,\rM)}
\end{align*}
of duplicial functors. Furthermore, $L
\circ R  = {\idty}$ if and only if $\rCC_\bT
(\rN,\rM)$ is cyclic, and $R  \circ L  = {\idty}$
if and only if $ \rCC_\bS^\op(\rN,\rM)$ is cyclic.
\end{prop}
\begin{proof}
This is verified by
straightforward computation.  However, it is convenient
to use a diagrammatic calculus as, {\em e.g.},~in
\cite{MR2415479}, in which natural transformations
$\rN\rV \rM \rightarrow \rN\rW\rM$ are visualised as string
diagrams, where $\rV$ and $\rW$ are words in $\rS,\rT$.  For
example $t^\bT$ will be represented by the
diagram
$$
	\STRINGDIAGRAM{ {\rN \atop } \FERMION[ddddd] &
	\rT \atop \FERMION[d] & \rT \atop \FERMION[d] &  \rT \atop
\FERMION[d] & \rT \atop \FERMION[d] & \rM \atop
\FERMION[dddddd] \\ & \FERMIONddddr & \FERMIONddddr
\ar@{}[r]|{\cdots} & \FERMIONddddr &  \BOSON[r]
\FERMIONddddlll & \\ &  &  &  & & \\ & & & & & & & \\
\FERMION[dd] & & &  & & &\\ & \BOSON[l] \FERMION[d] &
\FERMION[d] & \FERMION[d] \ar@{}[r]|\cdots  &
\FERMION[d] & & & \\ \atop \rN & \atop \rT & \atop \rT &
\atop \rT &  \atop \rT & \atop \rM } $$

Crossing of strings represents the distributive law $
\chi $ and the bosonic propagators represent the $ \chi
$-coalgebra structures $ \lambda \colon\rN\rS \rightarrow\rN\rT$
respectively $ \rho \colon \rT\rM \rightarrow \rS\rM$.

As a demonstration, the relation $R t^\bT=
t^\bS R$ for $n=2$ becomes
$$ \STRINGDIAGRAM{
	{\rN \atop } \FERMION[ddddd] & \rT \atop \FERMION[d] &
	\rT \atop \FERMION[d] & \rT \atop \FERMION[d] & \rM \atop
	\FERMION[dddddd] \\
	& \FERMIONddddr & \FERMIONddddr & \BOSON[r]
\FERMIONddddll& \\
	&  &  &  & & \\
	& & & & & & & \\
	\FERMION[dd] & & &  & & &\\
	& \BOSON[l] \FERMION[d] & \FERMION[d] & \FERMION[d]
& & & \\
\FERMION[ddddd] & \FERMION[d] & \FERMION[d] &
\FERMION[d] & \FERMION[d] \FERMION[dddddd] \\
	& \FERMIONddddr & \FERMIONddddr & \BOSON[r]
\FERMIONddddll& \\
	&  &  &  & & & = & &\\
	& & & & & & & \\
	\FERMION[dd] & & &  & & &\\
	\FERMION[ddddd] & \FERMION[ddddd] & \FERMIONddddr &
\FERMIONddddl \BOSON[r] & \FERMION[ddddd] & & \\
	& & & & & & & \\
	& & & & & & & \\
	& & & & & & & \\
	& & \FERMION[d] & \FERMION[d] \BOSON[r] & & & \\
	\atop \rN & \atop\rS & \atop\rS & \atop\rS & \atop \rM }
\STRINGDIAGRAM{
{\rN \atop } \FERMION[ddddd] & {\rT \atop }\FERMION[d]
& {\rT
\atop }\FERMION[d] &  {\rT \atop }\FERMION[d] & {\rM \atop
}\FERMION[d] \FERMION[dddddd] \\
	& \FERMIONddddr & \FERMIONddddr & \BOSON[r]
\FERMIONddddll& \\
	&  &  &  & & \\
	& & & & & & & \\
	\FERMION[dd] & & &  & & &\\
	\FERMION[ddddd] & \FERMION[ddddd] & \FERMIONddddr &
\FERMIONddddl \BOSON[r] & \FERMION[ddddd] & & \\
	& & & & & & & \\
	& & & & & & & \\
	& & & & & & & \\
	& & \FERMION[d] & \FERMION[d] \BOSON[r] & & & \\
	\FERMION[ddddd] & \FERMION[d] & \FERMION[d] &
\FERMION[d] & \FERMION[dddddd] \\
	& \BOSON[l] \FERMIONddddrr & \FERMIONddddl &
\FERMIONddddl& \\
	&  &  &  & & \\
	& & & & & & & \\
	\FERMION[dd] & & &  & & &\\
	& \FERMION[d] & \FERMION[d] & \FERMION[d] \BOSON[r]
& & \\ \atop \rN & \atop\rS & \atop\rS & \atop\rS & \atop \rM}
$$ which reflects the naturality of $ \lambda, \rho$,
and
$\chi$.  Analogously, the identities $Rd_i=d_iR$ and
$Rs_j=s_jR$ follow from the commutative diagrams in
Definition~\ref{coalgdef}, which are represented
diagrammatically by
$$
	\STRINGDIAGRAM{ \rT \atop
	\FERMION[dd] & \rM \atop \FERMION[dd] & & & & \rT \atop
\FERMION[dd] & \rM \atop \FERMION[dd]\\ & \BOSON[l] & & =
& & \\ \bullet & \atop \rM & & & & \bullet & \atop \rM}
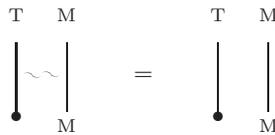
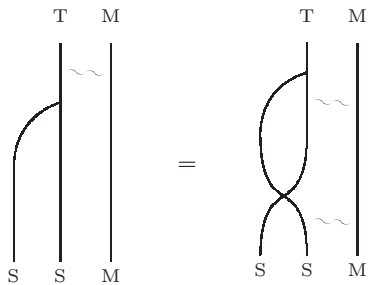
$$
respectively
$$
	\STRINGDIAGRAM{ & \rT \atop
\FERMION[dddd] & \rM \atop \FERMION[dddd]\\ & \BOSON[r] &
\\ & \bkomu[ddl]& \\ & & \\ \FERMION[ddd] &
\FERMION[ddd] & \FERMION[ddd] & & = & & \\  	& & \\ & &
\\ \atop\rS & \atop\rS & \atop \rM}
	\STRINGDIAGRAM{ & \rT
\atop \FERMION[ddd] & \rM \atop \FERMION[ddddddd]\\ &
\bkomu[ddl] & \\ & \BOSON[r] & \\ \FERMIONddddr &
\FERMIONddddl & \\ & & \\ & & \\ & \BOSON[r] & \\ \atop
\rS & \atop\rS & \atop \rM}
$$
Similarly, $L$ is a morphism
of duplicial objects, and one has $(L \circ R)_n =
(t_n^\bT)^{n+1}$ and $(R \circ L)_n =
(t_n^\bC)^{n+1}$.  \end{proof}

\subsection{The case of entwined coalgebras}
As we had announced above, entwined coalgebras
lead to trivial simplicial objects:

\begin{prop}
\label{trivcontract}
Let $\chi \colon
\bT\bS \rightarrow \bS\bT$ be a comonad distributive
law on a category $\cB$, and let $(\rM, \cX,
\rho)$ and $(\rN, \cY, \lambda)$ be left and right
$\chi$-coalgebras respectively. Suppose also that
$\cY$ is an abelian category.  If either of
$(\rN,\cY, \lambda),(\rM, \cX, \rho)$ is {entwined},
then the chain complexes associated to both
$\rCC_\bT (\rN, \rM)$ and $\rCC_\bS^\op (\rN, \rM)$
are contractible.
\end{prop}
\begin{proof}
If $(\rN, \cY, \lambda)$
is {entwined}, there is a $\bT$-opcoalgebra
structure $\nabla \colon \rN \rightarrow\rN\rT$ on $\rN$.
The morphisms $\nabla\rT^n \rM \colon\rN\rT^{n+1}\rM \rightarrow
\rN\rT^{n+2}\rM$ provide a contracting homotopy for the
complex associated to $\rCC_\bT (\rN, \rM)$,
and the morphisms
$$
	\xymatrix@=4em{ \rN\rS^{n+1} \rM
	\ar[r]^-{\nabla\rS^{n+1} \rM}
	&\rN\rT\rS^{n+1}\rM
	\ar[r]^-{\rN\chi^{n+1}\rM} &\rN\rS^{n+1} \rT \rM \ar[r]^-{\rN\rS^{n+1}
\rho} &\rN\rS^{n+2} \rM }
$$
provide a contracting homotopy
for the complex associated to $\rCC_\bS^\op
(\rN,\rM)$. The other case is similar.
\end{proof}

\subsection{Twisting by $1$-cells}
\label{twistsec}
In this
section, we show how factorisations of distributive
laws as considered in \cite{2} give rise to morphisms
between duplicial functors of the form considered
above.  To this end, fix a $1$-cell in the $2$-category
$\cDist$:
$$
\xymatrix@C=1.5em{ & \cB \ar@{}[r]^{\bT}="a"
\ar@{}[l]_{\bS}="b"  \ar[d]^-{(\Sigma, \sigma,
\gamma)}& \\ & \cD \ar@{}[r]_{\bG}="c"
\ar@{}[l]^{\bC}="d" & \ar @{.}@/_/ "a";"b"
|{\chi} \ar @{.}@/^/ "c";"d" |{\tau} } $$

\begin{lem}
\label{twist}
Let $(\rM, \cX, \rho)$ be
a right $\chi$-coalgebra. Then $(\Sigma \rM, \cX,
\gamma \rM \circ \Sigma \rho \circ \sigma \rM )$ is a right
$\tau$-coalgebra.
\end{lem}
\begin{proof}
This is
proved for the case that $\chi = \tau$ in \cite{2}, but the
same proof applies to this slightly more general
situation.
\end{proof}

Dually, left $\tau$-coalgebras $(\rN, \cY, \rho)$
define left $\chi$-coalgebras $(\rN\Sigma, \cY, \rN
\sigma \circ \lambda \Sigma \circ \rN \gamma)$.  The
following diagram illustrates the situation:
$$
\xymatrix@C=1.5em{ &&& \cB \ar@{.>}[llld]
\ar@{}[r]^{\bT}="a" \ar@{}[l]_{\bS}="b"
\ar[d]^-{(\Sigma, \sigma, \gamma)}& & &  \cX
\ar[lll]_-{(\rM, \rho)} \ar@{.>}[llld]\\ \cY &&&\ar[lll]^-{(\rN, \lambda)} \cD \ar@{}[r]_{\bG}
="c" \ar@{}[l]^{\bC}="d" & & & & \ar @{.}@/_/
"a";"b" |{\chi} \ar @{.}@/^/ "c";"d" |{\tau} } $$ The
dotted arrows represent the induced $\chi$-coalgebras
from Lemma~\ref{twist}.

Hence Theorem~\ref{dup} and Lemma~\ref{twist} yield
duplicial structures on the simplicial functors $$
\rCC_\bT(\rN\Sigma,\rM),\quad \rCC_\bS^\op(\rN\Sigma,\rM),
\quad \rCC_\bG(\rN,\Sigma \rM),
\quad \rCC_\bC^\op(\rN,\Sigma \rM),
$$
and from Proposition~\ref{cyc} we obtain morphisms
\begin{align*}
	&\xymatrix{ \rCC_\bT(\rN\Sigma,\rM)
	\ar[r]^{R^\chi}  &
	\rCC_\bS^\op(\rN\Sigma,\rM), } &
	\xymatrix{ \rCC_\bS^\op(\rN\Sigma,\rM)
	\ar[r]^{L^\chi} & \rCC_\bT(\rN\Sigma,\rM),} \\
	&\xymatrix{ \rCC_\bG(\rN,\Sigma \rM) \ar[r]^-{R^\tau}
& \rCC_\bC^\op(\rN,\Sigma \rM),} & \xymatrix{
\rCC_\bC^\op(\rN,\Sigma \rM) \ar[r]^-{L^\tau} &
\rCC_\bG(\rN,\Sigma \rM)}
\end{align*}
of duplicial
objects which determine the cyclicity of each functor.

Additionally, repeated application of $\sigma \colon \rG\Sigma
\rightarrow \Sigma \rT$ and $\gamma \colon \Sigma \rS
\rightarrow \rC \Sigma$ yields two duplicial morphisms
$$
\xymatrix{ \rCC_\bG(\rN,\Sigma \rM) \ar[r] &
\rCC_\bT(\rN\Sigma, \rM),} \quad \quad \quad
\xymatrix{ \rCC_\bS^\op(\rN\Sigma, \rM) \ar[r] &
\rCC_\bC^\op(\rN,\Sigma \rM).}
$$
Note that for
arbitrary functors $\rM$ and $\rN$ these are simplicial morphisms
which become duplicial morphisms if $\rM$ and $\rN$ have
coalgebra structures.

\section{Hopf monads and Hopf
algebroids}\label{brugsec}

\subsection{Opmodule adjunctions}\label{opmoduleadj}
One example of Theorem~\ref{arise} is provided by an
opmonoidal adjunction between monoidal categories:
\begin{defn}
An adjunction
$$
	\xymatrix{ (\cE,\otimes _\cE,\mathbf 1_\cE)
\ar@/^{0.5pc}/[rr]^-\rH \ar@{}[rr]|{\perp}&&
\ar@/^0.5pc/[ll]^-\rE (\cH,\otimes_\cH,\mathbf 1_\cH)
}
$$
between monoidal categories is \emph{opmonoidal}
if both $\rH$ and $\rE$ are opmonoidal functors.
\end{defn}

Some authors call these \emph{comonoidal adjunctions} or
\emph{bimonads}. Thus by definition, there are natural
transformations
$$
	\Xi \colon \rH(X \otimes_\cE Y)
	\rightarrow \rH X \otimes_\cH \rH Y,\quad
	\Psi \colon \rE(K \otimes_\cH L) \rightarrow
	\rE K \otimes_\cE \rE L,
$$
and $ \Psi$ is in fact an isomorphism, see
\cite{MR3020336,MR2793022,MR1942328,MR3175323,MR1887157}
for more information. It follows that
$$
	\rH(\mathbf 1_\cE) \otimes_\cH {-}
	\qquad \rE\rH(\mathbf 1_\cE)
	\otimes_\cE {-}
$$
form a compatible pair of comonads
as in Theorem~\ref{arise}
whose comonad structures are induced by the natural
coalgebra (comonoid) structures on $\mathbf 1_\cE$.

However, the examples we are more interested in arise
from \emph{opmodule adjunctions}
$$
	\xymatrix{ (\cA,\otimes _\cA)
\ar@/^{0.5pc}/[rr]^-{\rF} \ar@{}[rr]|{\perp}&&
\ar@/^0.5pc/[ll]^-{\rU} (\cB,\otimes_\cB)}
$$
over
$\xymatrix{\cE \ar@/^{0.3pc}/[r] & \ar@/^0.3pc/[l]
\cH,}$ {\em cf.}~\cite[Definition~4.1.1]{MR3020336}.  Here
$\cB$ is an $\cH$-module category with action
$\otimes_\cB \colon \cH \times \cB \rightarrow \cB$,
whereas $\cA$ is
an $\cE$-module category with action
$ \otimes_\cA \colon \cE
\times \cA \rightarrow \cA$, and there are natural
transformations
$$
	\Theta \colon \rF(Y \otimes_\cA Z) \rightarrow \rH Y
	\otimes_\cB \rF Z, \quad
	\Omega \colon \rU(L \otimes_\cB M) \rightarrow
	\rE L \otimes_\cA \rU M
$$
with $\Omega$ being an isomorphism
(see \cite[Proposition~4.1.2]{MR3020336}).

Now any coalgebra $C$ in $\cH$
defines a compatible pair of comonads
$$
	\rS=C \otimes_\cB -,\quad
	\rC=\rE C \otimes_\cA -
$$
on $\cB$
respectively $\cA$. It is such an
instance of Theorem~\ref{arise} that provides the
monadic generalisation of the setting from
\cite{MR2803876}, see Section~\ref{coeffsforhopf}.

\subsection{Bialgebroids and Hopf algebroids}
Opmonoidal
adjunctions can be seen as categorical
generalisations of bialgebras and more generally
(left) bialgebroids. We briefly recall the definitions
but refer to \cite{MR2553659,MR2803876} for further
details and references.

\begin{defn}
If $E$ is a $k$-algebra,
then an \emph{$E$-ring} is a $k$-algebra map $ \eta : E
\rightarrow H$.
\end{defn}

In particular, when $E=\Ae:=A \otimes_k \Aop$ is the
\emph{enveloping algebra} of a $k$-algebra $A$, then
$H$ carries two $A$-bimodule structures given by
$$
	a \lact h \ract
	b:=\eta(a \otimes_k b)h,\quad
	a \blact h \bract b:=
	h \eta(b \otimes_k a).
$$

\begin{defn}
A \emph{bialgebroid} is an $\Ae$-ring
$ \eta : \Ae \rightarrow H$
for which ${}_\lact H_\ract$
is a coalgebra in $(\Ae\cMod,\otimes_A,A)$
whose coproduct
$
	\Delta \colon H \rightarrow H_\ract \otimes_A
	{}_\lact H
$
satisfies
$$
	a \blact \Delta(h) =\Delta (h) \bract a,\quad
	\Delta
	(gh)=\Delta(g)\Delta(h),
$$
and whose counit
$	
\varepsilon \colon H \rightarrow A
$
defines a unital $H$-action on $A$ given by
$h(a):=\varepsilon (a \blact h)$.
\end{defn}

Finally, by a Hopf algebroid we mean \emph{left}
rather than \emph{full} Hopf algebroid, so there is
in general no antipode \cite{3}:

\begin{defn}[\cite{MR1800718}]
A \emph{Hopf algebroid} is a bialgebroid with bijective
\emph{Galois map}
$$
	\beta \colon {}_\blact H \otimes_\Aop
H_\ract  \rightarrow H_\ract \otimes_A {}_\lact H,\quad
g \otimes_\Aop h \mapsto \Delta (g)h.  $$ \end{defn}

As usual, we abbreviate
\begin{equation}
	\label{hunger}
	\Delta (h) =: h_{(1)} \otimes_A
	h_{(2)},\qquad
	\beta^{-1}(h \otimes_A 1) =: h_+
	\otimes_\Aop h_-.
\end{equation}

\subsection{The opmonoidal adjunction}\label{bimfromhopf}
Every $E$-ring $H$ defines a forgetful functor
$$
	\rE \colon H\cMod \rightarrow E\cMod
$$
with left
adjoint $\rH=H \otimes_E -$.
In the sequel, we abbreviate
$\cH:=H\cMod$ and $\cE:=E\cMod$.
If $H$ is a bialgebroid,
then $\cH$ is monoidal with tensor product
$K \otimes_\cH L$ of
two left $H$-modules $K$ and $L$ given by the tensor
product $K \otimes_A L$ of the underlying
$A$-bimodules whose $H$-module structure is given by
$$
	h(k \otimes_\cH l):=
	h_{(1)}(k) \otimes_A h_{(2)}(l).
$$
So by definition, we have $\rE(K \otimes _\cH L) =
\rE K \otimes_A \rE L$. The opmonoidal structure $\Xi$
on $\rH$ is defined by the map \cite{MR2793022,MR3020336}
\begin{align*}
	\rH(X \otimes_A Y)=H \otimes_\Ae (X \otimes_A Y)
	&\rightarrow \rH X \otimes_\cH \rH Y= (H \otimes_\Ae
	X) \otimes_A (H \otimes_\Ae Y), \\
	h \otimes_\Ae (x \otimes_A y) &\mapsto
	(h_{(1)} \otimes_\Ae x) \otimes_A
	(h_{(2)} \otimes_\Ae y).
\end{align*}

Schauenburg proved that this establishes a bijective
correspondence between bialgebroid structures on $H$
and monoidal structures on $H\cMod$
\cite[Theorem~5.1]{MR1629385}:
\begin{thm}
The following data are equivalent for an $\Ae$-ring
$ \eta \colon \Ae \rightarrow H$:
\begin{enumerate}
\item A bialgebroid structure on $H$.
\item A monoidal structure $(\otimes,\mathbf 1)$ on ${H\cMod}$
such that the adjunction
$$
	\xymatrix{({\Ae\cMod},\otimes_A,A) \hspace{-10mm}&
	\ar@/^2mm/[r] & \ar@/^2mm/[l] & \hspace{-10mm}
	({H\cMod},\otimes,\mathbf 1)}
$$
induced by $ \eta $ is
opmonoidal.
\end{enumerate}
\end{thm}

Consequently,
we obtain an opmonoidal monad
$$
	\rE\rH={}_\blact H_\bract
	\otimes_\Ae -
$$
on $\cE=\Ae\cMod$. This
takes the unit object $A$ to the cocentre $H
\otimes_\Ae A$ of the $A$-bimodule ${}_\blact
H_\bract$, and the comonad $\rH(\mathbf 1_\cE)
\otimes_\cE -$ is given by
$$
	(H \otimes_\Ae A) \otimes_A -,
$$
where the $A$-bimodule structure on the
cocentre is given by the actions
$\lact,\ract$ on $H$.

The lift to $\cH={H\cMod}$ takes a left
$H$-module $L$ to
$
	(H \otimes_\Ae A)
	\otimes_A L
$
with action
$$
	g ((h \otimes_\Ae 1) \otimes_A l)=
	(g_{(1)}h \otimes_\Ae 1) \otimes_A g_{(2)}l,
$$
and the
distributive law resulting from Theorem~\ref{arise} is
given by
$$
	\chi \colon g \otimes_\Ae ((h \otimes_\Ae 1)
	\otimes_A l) \mapsto (g_{(1)}h \otimes_\Ae 1) \otimes_A
(g_{(2)} \otimes_\Ae l).
$$
That is, it is the map
induced by the \emph{Yetter-Drinfel'd braiding}
$$
	H_\bract \otimes_A {}_\lact H \to H_\ract \otimes_A
	{}_\lact H	,\quad g \otimes_A h \mapsto g_{(1)}h
	\otimes_A g_{(2)}.
$$

For $A=k$, that is, when $H$ is a Hopf algebra, and
also trivially when $H=\Ae$, the monad and the comonad
on $\Ae\cMod$ coincide and are also a bimonad in the sense of
Mesablishvili and Wisbauer,
{\em cf.}~Section~\ref{wisbimonad}.  An example where the two
are different is the Weyl algebra, or more generally,
the universal enveloping algebra of a
Lie-Rinehart algebra \cite{MR1625610}. In these
examples, $A$ is commutative but
not central in $H$ in general, so ${}_\blact
H_\bract \otimes_\Ae -$ is different from $H_\ract
\otimes_A -$.

\subsection{Doi-Koppinen data}
The instance of Theorem~\ref{arise}
that we are most interested in is an opmodule
adjunction associated to the following structure:
\begin{defn}
A \emph{Doi-Koppinen datum} is a triple
$(H,C,B)$ of an $H$-module coalgebra $C$ and an $H$-comodule
algebra $B$ over a bialgebroid $H$.
\end{defn}

This means that $C$ is a coalgebra in
the monoidal category ${H\cMod}$.
Dually, the category $H\cComod$ of left $H$-comodules is
also monoidal, and this defines the notion of a
comodule algebra. Explicitly, $B$ is an $A$-ring
$
	\eta_B \colon A \rightarrow B
$
together with a coassociative coaction
$$
	\delta \colon B \rightarrow H_\ract \otimes_A
B,\quad b
	\mapsto b_{(-1)} \otimes_A b_{(0)},
$$
which is counital and an algebra map,
$$
	\eta _B (\varepsilon
	(b_{(-1)}))b_{(0)}=b,\quad
	(b d)_{(-1)} \otimes (b d)_{(0)} =
	b_{(-1)} d_{(-1)}
	\otimes b_{(0)} d_{(0)}.
$$
Similarly, as in the
definition of a bialgebroid itself, for this condition
to be well-defined one must also require
$$
	b_{(-1)}
	\otimes_A b_{(0)} \eta_B(a)= a \blact b_{(-1)}
	\otimes_A b_{(0)}.
$$

The key example that reproduces
\cite{MR2803876} is the following:

\subsection{The opmodule
adjunction}\label{associatedsec}
For any Doi-Koppinen datum $(H,C,B)$, the
$H$-coaction $ \delta $ on $B$ turns the Eilenberg-Moore adjunction
\!\!
$\xymatrix{{A\cMod} \ar@/^{0.3pc}/[r] &
\ar@/^0.3pc/[l] {B\cMod}}$
\!\!
for the monad $ \rB:=B \otimes_A - $ into an opmodule
adjunction for the opmonoidal adjunction
$\xymatrix{\cE \ar@/^{0.3pc}/[r]
& \ar@/^0.3pc/[l] \cH}$ defined in
Section~\ref{bimfromhopf}.  The $\cH$-module
category structure of $B\cMod$ is given by the left
$B$-action
$$
	b(l \otimes_A m):=
	b_{(-1)}l \otimes_A
	b_{(0)}m,
$$
where $b \in B$,
$l \in L$ (an $H$-module), and $m \in
M$ (a $B$-module).

Hence, as explained in Section~\ref{opmoduleadj},
$C$ defines a compatible pair of
comonads $ C \otimes_A - $ on ${B\cMod}$ and ${A\cMod}$. The
distributive law resulting from Theorem~\ref{arise}
generalises the Yetter-Drinfel'd braiding, as it is
given for a $B$-module $M$ by
\begin{eqnarray*}
	\chi \colon B \otimes_A (C
	\otimes_A M) &\rightarrow& C \otimes_A (B \otimes_A M),
\\
	b \otimes_A (c \otimes_A m) &\mapsto& b_{(-1)}c
\otimes _A (b_{(0)} \otimes_A m).
\end{eqnarray*}

\subsection{The main example}\label{coeffsforhopf}
If $H$ is a bialgebroid, then
$C:=H$ is a module coalgebra with
left action given by multiplication and coalgebra
structure given by that of $H$.
If $H$
is a Hopf algebroid, then $B:=H^\mathrm{op}$ is
a comodule algebra with unit map $\eta_B(a):=\eta(1
\otimes_k a)$ and coaction
$$
 	\delta \colon H^\mathrm{op} \rightarrow
	H_\ract \otimes_A {}_\blact H^\mathrm{op}, \quad
 b \mapsto b_- \otimes_A b_+.
$$
In the sequel we write $\rB$ as $-
\otimes_\Aop H$ rather than $H^\mathrm{op} \otimes_A -$
to work with $H$ only.  Then the distributive law
becomes
\begin{eqnarray*}
\nonumber
	\chi \colon (H \otimes_A M) \otimes_\Aop H
&\rightarrow& H \otimes_A (M \otimes_\Aop H),
\\
(c \otimes_A m) \otimes_\Aop b &\mapsto& b_- c \otimes _A (m
\otimes_\Aop b_+),
\end{eqnarray*}
for $b, c \in H$.

Proposition~\ref{chicoalgprop} completely
characterises the right $ \chi $-coalgebras: in this
example, they are given by right $H$-modules and left
$H$-comodules $M$ with right $ \chi $-coalgebra
structure
$$
	\rho: m \otimes_\Aop h \mapsto
	h_-m_{(-1)}
	\otimes_A m_{(0)} h_+.
$$
Recall furthermore that there is no analogue
of Proposition~\ref{chicoalgprop} for left
$ \chi $-coalgebras. However, the specific example of a Hopf
algebroid might provide some indication towards such a
result. Indeed, here one can carry out an analogous
construction of left $ \chi $-coalgebras associated to
(left-left) Yetter-Drinfel'd modules:

\begin{defn}
A \emph{Yetter-Drinfel'd module} over $H$ is a left
$H$-comodule and left $H$-module $N$ such that
for all $ h\in H,n \in N$, one has
$$
	(hn)_{(-1)} \otimes_A (hn)_{(0)}=
	h_{+(1)} n_{(-1)} h_{-} \otimes_A
	h_{+(2)}n_{(0)}.
$$
\end{defn}

Each such Yetter-Drinfel'd
module defines a left $ \chi$-coalgebra
$$
	\rN := - \otimes_H N \colon H^\mathrm{op}\cMod \rightarrow
	k\cMod
$$
whose $ \chi $-coalgebra structure is given by
\begin{equation*}
\lambda: (h \otimes_A x) \otimes_H n
	\mapsto (xn_{(-1)+} h_+ \otimes_\Aop h_- n_{(-1)-})
	\otimes_H n_{(0)}.
\end{equation*}
The resulting duplicial object
$\rCC_\bT(\rN,\rM)$ is the one
studied in \cite{MR2803876, Kow:GABVSOMOO}.

Identifying
$(- \otimes_\Aop H) \otimes_H N \cong - \otimes_\Aop N$,
the $ \chi $-coalgebra structure
  becomes
\begin{equation*}
\lambda: (h \otimes_A x) \otimes_H n
	\mapsto
	xn_{(-1)+} h_+ \otimes_\Aop h_- n_{(-1)-} n_{(0)}.
\end{equation*}
Using this identification, we give
explicit expressions of the operators $L_n$ and $R_n$
as well as $t^\mathbb{T}_n$ that appeared in
Sections~\ref{evidenziatore1} and~\ref{evidenziatore2}:
first of all,
observe that the right $H$-module structure on
$
	\rS M := H_\ract \otimes_A M
$
is given by
$$
	(h \otimes_A m)g :=
	g_- h \otimes_A mg_+,
$$
whereas the right $H$-module structure on
$
\rT M := M \otimes_\Aop H_\ract
$
is given by
$$
	(m \otimes_\Aop h)g :=
	m \otimes_\Aop hg.
$$
The cyclic operator from Section \ref{evidenziatore1}
then results as
\begin{equation*}
\begin{split}
&t^\mathbb{T}_n
(m  \otimes_\Aop h^1 \otimes_\Aop \cdots \otimes_\Aop h^n \otimes_\Aop n)
\\
&=
m_{(0)} h^1_+ \otimes_\Aop h^2_+ \otimes_\Aop \cdots \otimes_\Aop h^n_+ \\
& \qquad
\otimes_\Aop (n_{(-1)} h^n_- \cdots h^1_- m_{(-1)})_+
\otimes_\Aop (n_{(-1)} h^n_- \cdots h^1_- m_{(-1)})_-
n_{(0)},
\end{split}
\end{equation*}
and for the operators $L$ and $R$ from Section \ref{evidenziatore2}
 one obtains with the help of the properties
\cite[Prop.~3.7]{MR1800718} of the translation map
\rmref{hunger}:
\begin{equation*}
\begin{split}
L_n:  (h^1 & \otimes_A \cdots \otimes_A h^{n+1}
\otimes_A m) \otimes_H n
\mapsto \\
& (mn_{(-1)+} h^1_+ \otimes_\Aop
h^1_- h^2_+ \otimes_\Aop \cdots \otimes_\Aop h^{n+1}_- n_{(-1)-})
\otimes_H n_{(0)},
\end{split}
\end{equation*}
along with
\begin{equation*}
\begin{split}
R_n: (m & \otimes_\Aop h^1
\otimes_\Aop \cdots \otimes_\Aop h^n \otimes_\Aop 1)
\otimes_H n
\mapsto \\
&
(m_{(-n-1)}  \otimes_A m_{(-n)}h^1_{(1)} \otimes_A
m_{(-n+1)}h^1_{(2)}h^2_{(1)}  \otimes_A \cdots \\
& \quad
\otimes_A m_{(-1)} h^1_{(n)}
h^2_{(n-1)} \cdots h^n_{(1)}
\otimes_A m_{(0)}) \otimes_H h^1_{(n+1)} h^2_{(n)} \cdots h^n_{(2)} n.
\end{split}
\end{equation*}
Compare these maps with those obtained in \cite[Lemma 4.10]{MR2803876}.
Hence, one has:

\begin{equation*}
\begin{split}
(L_n &\circ R_n)\big((m  \otimes_\Aop h^1
\otimes_\Aop \cdots \otimes_\Aop h^n \otimes_\Aop 1)
	\otimes_H n\big)
= \\
&
m_{(0)} (h^1_{(n+1)}h^2_{(n)} \cdots h^n_{(2)} n)_{(-1)+} m_{(-n-1)+}
\otimes_\Aop m_{(-n-1)-}  m_{(-n)+}  h^1_{(1)+}
\\
&
\qquad
\otimes_\Aop
h^1_{(1)-} m_{(-n)-}  m_{(-n+1)+}  h^1_{(2)+} h^2_{(1)+} \otimes_\Aop
\cdots
\\
&
\qquad
\otimes_\Aop h^n_{(1)-} \cdots h^1_{(n)-} m_{(-1)-} (h^1_{(n+1)} \cdots h^n_{(2)} n)_{(-1)-} (h^1_{(n+1)} \cdots h^n_{(2)} n)_{(0)}
\\
&
=
m_{(0)} \big((h^1_{(2)} \cdots h^n_{(2)} n)_{(-1)} m_{(-1)}\big)_+
\otimes_\Aop h^1_{(1)+} \otimes_\Aop \cdots
\\
&
\quad
\otimes_\Aop h^n_{(1)+}
\otimes_\Aop h^n_{(1)-} \cdots h^1_{(1)-} \big((h^1_{(2)} \cdots h^n_{(2)} n)_{(-1)} m_{(-1)} \big)_- (h^1_{(2)} \cdots h^n_{(2)} n)_{(0)}.
\end{split}
\end{equation*}
Finally, if $M \otimes_\Aop N$ is a stable anti
Yetter-Drinfel'd module \cite{MR2415479},
that is, if
$$
m_{(0)}(n_{(-1)}m_{(-1)})_+ \otimes_\Aop (n_{(-1)}m_{(-1)})_- n_{(0)} = m \otimes_\Aop n
$$
holds for all $n \in N$, $m \in M$,
we conclude by
\begin{equation*}
\begin{split}
(L_n \circ R_n)(m & \otimes_\Aop h^1 \otimes_\Aop \cdots
\otimes_\Aop h^n \otimes_\Aop n)
\\
&
=
m \otimes_\Aop h^1_{(1)+} \otimes_\Aop \cdots \otimes_\Aop h^n_{(1)+}
\otimes_\Aop h^n_{(1)-} \cdots h^1_{(1)-} h^1_{(2)} \cdots h^n_{(2)} n
\\
&
=
m \otimes_\Aop h^1 \otimes_\Aop \cdots \otimes_\Aop h^n
\otimes_\Aop n.
\end{split}
\end{equation*}
Observe that in \cite{Kow:GABVSOMOO} this
cyclicity condition was obtained for a different complex
which, however, computes the same homology.

\subsection{The antipode as a $1$-cell}
\label{viviverde}
If $A=k$, then the four actions
$\lact,\ract,\blact,\bract$ coincide and
$H$ is a Hopf algebra with antipode
$S \colon H \rightarrow H$ given by
$S(h)=\varepsilon (h_+)h_-$. The
aim of this brief section is to remark that
this defines a 1-cell that connects the two instances
of Theorem~\ref{arise} provided by the opmonoidal
adjunction and the opmodule adjunction considered
above.

Indeed, in this case we have $\Ae\cMod \cong A\cMod=k\cMod$,
but $H^\mathrm{op}\cMod \neq H\cMod$ unless $H$ is
commutative. However, $S$ defines
a lax morphism $\sigma \colon
- \otimes_k H \, \idty \rightarrow H \otimes_k - \,{\idty}$,
given in components by
$$
		\sigma_X \colon X \otimes_k H
	\rightarrow H \otimes_k X, \quad
	x \otimes_k h \mapsto S(h) \otimes_k x.
$$
The fact that this is a lax morphism is
equivalent to the fact that $S$ is an algebra
anti-homomorphism.
Also, the lifted comonads agree and are given by
$H \otimes_k -$ with comonad structure given by the
coalgebra structure of $H$;
clearly, $\gamma = {\idty}
\colon {\idty} H \otimes_k - \rightarrow H \otimes_k - {\idty}$
is a colax morphism. Furthermore, the Yang-Baxter
condition is satisfied, so we have that $({\idty},
\sigma, \gamma)$ is a $1$-cell in the $2$-category of mixed
distributive laws. If we apply the $2$-functor $i$ to
this, we get a $1$-cell $(\Sigma, \tilde \sigma, \tilde
\gamma)$ between a comonad distributive law on the
category of left $H$-modules and one on the category of
right $H$-modules. The identity lifts to the functor
$\Sigma \colon H\mbox{-}\mathsf{Mod} \rightarrow
\mathsf{Mod}\mbox{-}H$ which sends a left $H$-module
$X$ to the right $H$-module with right action given by
$$
x \leftslice h := S(h) x.
$$

\section{Hopf monads \`a la Mesablishvili-Wisbauer}
\label{wisbimonad}
\subsection{Bimonads}
A \emph{bimonad} in the sense of
\cite{MR2787298} is a sextuple
$(\rA,\mu,\eta,\Delta^\rA,\varepsilon^\rA,\theta)$,
where $\rA
\colon \cC \rightarrow \cC$ is a functor, $(\rA,\mu,\eta)$ is
a monad, $(\rA,\Delta^\rA,\varepsilon^\rA)$ is a comonad and
$\theta \colon \rA\rA \rightarrow \rA\rA$ is a mixed distributive
law satisfying a list of compatibility conditions.

In particular, $ \mu $ and $ \Delta^A $ are required to
be compatible in the sense that there is a commutative
diagram \begin{equation}\label{wisbauerdiagram}
\begin{array}{c} \xymatrix{\rA\rA \ar[d]_{\rA
{\Delta}^\rA }
\ar[r]^\mu & \rA \ar[r]^{\Delta^\rA} & \rA\rA\\
\rA\rA\rA
\ar[rr]_{\theta \rA} & & \rA\rA\rA \ar[u]_{\rA\mu}} \end{array}
\end{equation}
The other defining conditions rule the
compatibility between the unit and the counit with each
other and with $ \mu $ respectively $ \Delta^\rA $, see
\cite{MR2787298} for the details.

It follows immediately that we also obtain an instance
of Theorem~\ref{arise} in this situation: if we take
$\cA=\cC^\bB$ to be the Eilenberg-Moore
category of the monad $\bB=(\rA,\mu,\eta)$ as in
Section~\ref{exceptional}, then the mixed distributive
law $\theta$ defines a lift
$\bV=(\rV,\Delta^\rV,\varepsilon^\rV)$ of the comonad
$\bC=(\rA,\Delta^\rA,\varepsilon^\rA)$ to $\cA$.

Note that in general, neither $\cA$ nor $\cC$ need to
be monoidal, so $\rB$ is in general not an opmonoidal
monad. Conversely, recall that for the examples of
Theorem~\ref{arise} obtained from opmonoidal monads,
$\rB$ need not equal $\rC$.

\subsection{Examples from bialgebras} In the main
example of bimonads in the above sense, we in fact do
have $\rB=\rC$ and we are in the situation of
Section~\ref{bimfromhopf} for a bialgebra $H$ over
$A=k$.  The commutativity of (\ref{wisbauerdiagram})
amounts to the fact that the coproduct is an algebra
map.

This setting provides an instance of
Proposition~\ref{sunshines} since there are two lifts
of $\rB=\rC$ from $\cA={k\cMod}$ to $\cB={H\cMod}$: the
canonical lift $\rS=\rT=\rF\rU$ which takes a left
$H$-module $L$ to the $H$-module $H \otimes_k L$ with
$H$-module structure given by multiplication in the
first tensor component, and the lift $\rV$ which takes
$L$ to $H \otimes_k L$ with $H$-action given by the
codiagonal action
$
	g(h \otimes_k y)=
	g_{(1)}h \otimes_k
	g_{(2)}y,
$
that is, the one defining the monoidal
structure on $\cB$.
Now the Galois map from Proposition~\ref{wisga} is the
Galois map
$$
	H \otimes_k L \rightarrow H \otimes_k
	L,\quad
	g\otimes_k y \mapsto g_{(1)} \otimes_k
	g_{(2)}y
$$
used to define left Hopf algebroids (when taking tensor products over $A \neq k$ resp.\ $\Aop$), which
for $A=k$ are simply Hopf algebras, and more generally
Hopf monads in the sense of
\cite[Theorem~5.8(c)]{MR3320218}.

\subsection{An example not from bialgebras}
Another
example of a bimonad is the \emph{nonempty list monad}
$\bL^+$ on $\cSet$, which assigns to a
set $X$ the set $\rL^+X$ of all nonempty lists
of elements in $X$, denoted $[x_1, \ldots, x_n]$. The monad
multiplication is given by concatenation of lists and
the unit maps $x$ to $[x]$.  The comonad
comultiplication is given by $\Delta [x_1, \ldots, x_n]
= [[x_1, \ldots, x_n], \ldots, [x_n]]$, the counit
is $\varepsilon[x_1, \ldots, x_n] = x_1$, and the mixed
distributive law
$$ \theta \colon \bL^+ \bL^+ \rightarrow \bL^+ \bL^+ $$ is defined as
follows: given a list
$$ [ [ x_{1,1}, \ldots, x_{1,
n_1} ] , \ldots, [x_{m,1}, \ldots, x_{m, n_m}]] $$ in
$\rL^+ X$, its image under $\theta X$ is the list with $$
\sum_{i=1}^m n_i (m-i+1) $$ terms, given by the
lexicographic order, that is \begin{align*}
\Big[[x_{1,1}, x_{2,1}, x_{3,1} \ldots, x_{m,1}],
\ldots, [x_{1, n_1}, x_{2,1}, x_{3,1}, \ldots,
x_{m,1}]&, \\ [x_{2,1}, x_{3,1} \ldots, x_{m,1}],
\ldots, [x_{2, n_2}, x_{3,1}, \ldots x_{m,1} ]&,\\
\ldots&,  \\  [ x_{m,1} ] , [ x_{m,2} ], \ldots,
[x_{m,n_m} ] \Big].&
\end{align*}

One verifies straightforwardly:

\begin{prop} $\bL^+$ becomes a bimonad on $\mathsf
{Set}$ whose Eilenberg-Moore category is
$\cSet^{\bL^+} \cong \mathsf{SemiGp}$, the
category of (nonunital) semigroups.  \end{prop}

The second lift $\rV$ of the comonad $\bL^+$ that
one obtains from the bimonad structure on
$\mathsf{SemiGp}$ is as follows. Given a semigroup $X$,
we have $\rV X = \rL^+ X$ as sets, but the binary
operation is given by \begin{align*} \rV X \times \rV X
&\rightarrow \rV X \\ [x_1, \ldots, x_m][y_1, \ldots, y_n] &:=
[x_1y_1, \ldots, x_my_1, y_1, \ldots, y_n]. \end{align*}

Following Proposition~\ref{chicoalgprop}, given a
semigroup $X$, the unit turns the underlying set of $X$ into an
$\bL^+$-coalgebra and hence we get a right $
\chi$-coalgebra structure on $X$. Explicitly, $\rho_X \colon \rT X \rightarrow \rV X$ is given by
$$
\rho[ x_1, \ldots, x_n] = [x_1 \cdots x_n, x_2 \cdots
x_n, \ldots, x_n].
$$
The image of $\rho$ is known as
the \emph{left machine expansion} of $X$
\cite{MR745358}.

\begin{prop}
The only $\theta$-entwined algebra is the trivial semigroup $\emptyset$.
\end{prop}
\begin{proof}
An $\bL^+$-coalgebra structure $\beta \colon T \to \bL^+ T$ is equivalent to $T$ being a forest of at most countable height (rooted) trees, where each level may have arbitrary cardinality. The structure map $\beta$ sends $x$ to the finite list of predecessors of $x$. A $\theta$-entwined algebra is therefore such a forest, which also has the structure of a semigroup such that for all $x,y \in T$ with $\beta(y) = [y, y_1, \ldots, y_n]$ we have
$$
\beta(xy) = [xy, xy_1, \ldots, xy_n , y, y_1, \ldots, y_n].
$$

Let $T$ be a $\theta$-entwined algebra. If $T$ is non-empty, then there must be a root. We can multiply this root
with itself to generate branches of arbitrary height. Suppose that we have a
branch of height two; that is to say, an element $y \in T$ with $\beta(y) = [y,x]$ 
(so, in particular, $x \neq y$). Then $\beta(xy) = [xy, y]$,  but
$\beta(xx) = [xx, xy, x, y]$. This is impossible since $x$ and $y$ cannot both be the predecessor of $xy$.
\end{proof}


\begin{thebibliography}{LMW15}

\bibitem[AC12]{MR3020336}
M.~Aguiar and S.~U.~Chase, \emph{Generalized {H}opf modules for
  bimonads}, Theory Appl. Categ. \textbf{27} (2012), 263--326. 

\bibitem[App65]{MR2614998}
H.~W.~Applegate, \emph{Acyclic models and resolvent functors},
Ph.~D.~Thesis Columbia University, ProQuest LLC, Ann Arbor, MI, 1965.


\bibitem[BLV11]{MR2793022}
A.~Brugui{\`e}res, S.~Lack, and A.~Virelizier, \emph{Hopf monads on
  monoidal categories}, Adv. Math. \textbf{227} (2011), no.~2, 745--800.

\bibitem[BM98]{MR1604340}
T.~Brzezi{\'n}ski and S.~Majid, \emph{Coalgebra bundles}, Comm. Math.
  Phys. \textbf{191} (1998), no.~2, 467--492. 

\bibitem[B{\"o}h09]{MR2553659}
G.~B{\"o}hm, \emph{Hopf algebroids}, Handb. Algebr., vol.~6, Elsevier/North-Holland, Amsterdam, 2009,
  pp.~173--235. 

\bibitem[BR84]{MR745358}
J.-C.~Birget and J.~Rhodes, \emph{Almost finite expansions of
  arbitrary semigroups}, J. Pure Appl. Algebra \textbf{32} (1984), no.~3,
  239--287. 

\bibitem[B{\c{S}}08]{MR2415479}
G.~B{\"o}hm and D.~{\c{S}}tefan, \emph{({C}o)cyclic
  (co)homology of bialgebroids: an approach via (co)monads}, Comm. Math. Phys.
  \textbf{282} (2008), no.~1, 239--286. 

\bibitem[B{\c{S}}12]{MR2956318}
\bysame, \emph{A categorical approach to cyclic duality}, J. Noncommut. Geom.
  \textbf{6} (2012), no.~3, 481--538. 

\bibitem[Bur73]{MR0323864}
{\'E.}~Burroni, \emph{Lois distributives mixtes}, C. R. Acad. Sci. Paris
  S\'er. A-B \textbf{276} (1973), A897--A900. 

\bibitem[CM98]{MR1657389}
A.~Connes and H.~Moscovici, \emph{Hopf algebras, cyclic cohomology and the
  transverse index theorem}, Comm. Math. Phys. \textbf{198} (1998), no.~1,
  199--246. 


\bibitem[Con83]{Con:CCEFE}
A.~Connes, \emph{Cohomologie cyclique et foncteurs {${\rm Ext}\sp n$}}, C. R.
  Acad. Sci. Paris S\'er. I Math. \textbf{296} (1983), no.~23, 953--958.

\bibitem[Con85]{MR823176}
\bysame, \emph{Noncommutative differential geometry}, Inst. Hautes
  \'Etudes Sci. Publ. Math. (1985), no.~62, 257--360. 


\bibitem[DK85]{MR826872}
W.~G.~Dwyer and D.~M.~Kan, \emph{Normalizing the cyclic modules of {C}onnes},
  Comment. Math. Helv. \textbf{60} (1985), no.~4, 582--600. 

\bibitem[DK87]{MR885102}
\bysame, \emph{Three homotopy theories for cyclic modules}, Proceedings of the
  {N}orthwestern conference on cohomology of groups ({E}vanston, {I}ll., 1985),
  vol.~44, 1987, pp.~165--175. 

\bibitem[Hue98]{MR1625610}
J.~Huebschmann, \emph{Lie-{R}inehart algebras, {G}erstenhaber algebras
  and {B}atalin-{V}ilkovisky algebras}, Ann. Inst. Fourier (Grenoble)
  \textbf{48} (1998), no.~2, 425--440. 

\bibitem[Joh75]{MR0390018}
P.~T. Johnstone, \emph{Adjoint lifting theorems for categories of algebras},
  Bull. London Math. Soc. \textbf{7} (1975), no.~3, 294--297. 

\bibitem[Kas87]{MR883882}
C.~Kassel, \emph{Cyclic homology, comodules, and mixed complexes}, J.
  Algebra \textbf{107} (1987), no.~1, 195--216. 

\bibitem[KK11]{MR2803876}
N.~Kowalzig and U.~Kr{\"a}hmer, \emph{Cyclic structures in algebraic
  (co)homology theories}, Homology Homotopy Appl. \textbf{13} (2011), no.~1,
  297--318. 

\bibitem[Kow13]{Kow:GABVSOMOO}
N.~Kowalzig, \emph{Gerstenhaber and Batalin-Vilkovisky structures on modules over operads}, (2013), preprint, {\tt arXiv:1312.1642}, to appear in Int.~Math.~Res.~Not.

\bibitem[KMT03]{MR1943179}
J.~Kustermans, G.~J. Murphy, and L.~Tuset, \emph{Differential calculi over
  quantum groups and twisted cyclic cocycles}, J. Geom. Phys. \textbf{44}
  (2003), no.~4, 570--594. 

\bibitem[KR13]{3}
U.~Kr{\"a}hmer and A.~Rovi, \emph{A Lie-Rinehart algebra with no antipode}, (2013),
preprint, {\tt arXiv:1308.6770}, to appear in Comm.~Algebra.

\bibitem[KS14]{2}
U.~Kr{\"a}hmer and P.~Slevin, \emph{Factorisations of distributive laws}, (2014),
preprint, {\tt arXiv:1409.7521}.

\bibitem[Lei04]{MR2094071}
Tom Leinster, \emph{Higher operads, higher categories}, London Mathematical
  Society Lecture Note Series, vol. 298, Cambridge University Press, Cambridge,
  2004. 

\bibitem[LMW15]{MR3320218}
M.~Livernet, B.~Mesablishvili, and R.~Wisbauer, \emph{Generalised
  bialgebras and entwined monads and comonads}, J. Pure Appl. Algebra
  \textbf{219} (2015), no.~8, 3263--3278. 

\bibitem[McC02]{MR1942328}
P.~McCrudden, \emph{Opmonoidal monads}, Theory Appl. Categ. \textbf{10}
  (2002), No. 19, 469--485. 

\bibitem[ML98]{MR1712872}
S.~Mac~Lane, \emph{Categories for the working mathematician}, second ed.,
  Graduate Texts in Mathematics, vol.~5, Springer-Verlag, New York, 1998.

\bibitem[MM02]{MR1877862}
C.~Menini and G.~Militaru, \emph{Integrals, quantum {G}alois extensions, and
  the affineness criterion for quantum {Y}etter-{D}rinfel\cprime d modules}, J.
  Algebra \textbf{247} (2002), no.~2, 467--508. 

\bibitem[Moe02]{MR1887157}
I.~Moerdijk, \emph{Monads on tensor categories}, J. Pure Appl. Algebra
  \textbf{168} (2002), no.~2-3, 189--208, Category theory 1999 (Coimbra).

\bibitem[MS99]{MR1710737}
E.~F.~M{\"u}ller and H.-J.~Schneider, \emph{Quantum homogeneous spaces with
  faithfully flat module structures}, Israel J. Math. \textbf{111} (1999),
  157--190. 

\bibitem[MW10]{MR2651345}
B.~Mesablishvili and R.~Wisbauer, \emph{Galois functors and entwining
  structures}, J. Algebra \textbf{324} (2010), no.~3, 464--506. 

\bibitem[MW11]{MR2787298}
\bysame, \emph{Bimonads and {H}opf monads on categories}, J. K-Theory
  \textbf{7} (2011), no.~2, 349--388. 

\bibitem[MW14]{MR3175323}
\bysame, \emph{Galois functors and generalised {H}opf modules}, J. Homotopy
  Relat. Struct. \textbf{9} (2014), no.~1, 199--222. 

\bibitem[Sch98]{MR1629385}
P.~Schauenburg, \emph{Bialgebras over noncommutative rings and a structure
  theorem for {H}opf bimodules}, Appl. Categ. Structures \textbf{6} (1998),
  no.~2, 193--222. 

\bibitem[Sch00]{MR1800718}
\bysame, \emph{Duals and doubles of quantum groupoids ({$\times_R$}-{H}opf
  algebras)}, New trends in {H}opf algebra theory ({L}a {F}alda, 1999),
  Contemp. Math., vol. 267, Amer. Math. Soc., Providence, RI, 2000,
  pp.~273--299. 

\bibitem[Str72]{MR0299653}
R.~Street, \emph{The formal theory of monads}, J. Pure Appl. Algebra
  \textbf{2} (1972), no.~2, 149--168. 

\bibitem[Wei94]{MR1269324}
C.~A.~Weibel, \emph{An introduction to homological algebra}, Cambridge
  Studies in Advanced Mathematics, vol.~38, Cambridge University Press,
  Cambridge, 1994. 

\end{thebibliography}
\bibliographystyle{amsalpha}

\def\cprime{$'$}
\providecommand{\bysame}{\leavevmode\hbox to3em{\hrulefill}\thinspace}
\providecommand{\MR}{\relax\ifhmode\unskip\space\fi MR }
\providecommand{\MRhref}[2]{%
  \href{http://www.ams.org/mathscinet-getitem?mr=#1}{#2}
}
\providecommand{\href}[2]{#2}

\end{document}